\documentclass[11pt]{article}
\usepackage{bbm}
\usepackage{psfrag,epic,eepic,epsfig}
\usepackage{fullpage,color}
\usepackage[applemac]{inputenc} 
\usepackage{amsmath,amsfonts,amssymb,amsthm,mathrsfs}
\usepackage[font=sf, labelfont={sf,bf}, margin=1cm]{caption}
\usepackage{graphicx,graphics}
\usepackage{latexsym}
\usepackage{ae,aecompl}
\usepackage[english]{babel}
 \usepackage[colorlinks=true]{hyperref}
\usepackage{enumerate}
\usepackage[normalem]{ulem}
\usepackage{xcolor}
\usepackage{empheq}
\usepackage[framemethod=tikz]{mdframed}
\usepackage{lipsum}
\usepackage[T1]{fontenc}
\usepackage{babel}
\usepackage{tikz}

\definecolor{gris25}{gray}{0.75}

\newtheorem{thm}{Theorem}[section]
\newtheorem{defn}[thm]{Definition}
\newtheorem{lem}[thm]{Lemma}
\newtheorem{prop}[thm]{Proposition}
\newtheorem{cor}[thm]{Corollary}

\newcommand{\T}{\mathcal{T}}

\newmdenv[innerlinewidth=0.5pt, roundcorner=4pt,innerleftmargin=6pt,
innerrightmargin=6pt,innertopmargin=6pt,innerbottommargin=6pt]{mybox3}

\date{}

\begin{document}

\title{\textbf{{\Huge Scaling limits of Markov--Branching trees and applications}} \\
{\large Lecture notes of the XII Simposio de Probabilidad y Procesos Estoc\'asticos}
\\
{\large 16 -- 20 novembre 2015, Mérida, Yucat\'an}
}

\author{\Large Bénédicte Haas\thanks{Université Paris 13, Sorbonne Paris Cité, LAGA, CNRS (UMR  7539) 93430 Villetaneuse, France. \newline E--mail:  \textcolor{red}{haas@math.univ-paris13.fr}. \newline This work was partially supported by the ANR GRAAL ANR--14--CE25--0014.}}

\maketitle

\abstract{The goal of these lectures is to survey some of the recent progress on the description of large--scale structure of random trees. We use the framework of Markov--Branching sequences of trees  and discuss several applications.

\tableofcontents

\setlength{\parindent}{0pt}
\setlength{\parskip}{8pt}

\section{Introduction}

The goal of these lectures is to survey some of the recent progress on the description of large--scale structure of random trees. Describing the structure of large (random) trees, and more generally large graphs, is an important goal of modern probabilities and combinatorics. Beyond the purely probabilistic or combinatorial aspects, motivations come from the study of models from biology, theoretical computer science or mathematical physics.

The question we will typically be interested in is the following. For $(T_n,n \geq 1)$ a sequence of random trees, where, for each $n$, $T_n$ is a tree of size $n$ (the size of a tree may be its number of vertices or its number of leaves,  for example):
does there exist a deterministic sequence $(a_n,n\geq 1)$ and a \emph{continuous} random tree $\mathcal T$ such that
$$
\frac{T_n}{a_n} \underset{n \rightarrow \infty}\longrightarrow \mathcal T \hspace{0.02cm}?
$$
To make sense of this question, we will view $T_n$ as a metric space by ``replacing" its edges with segments of length 1, and then use the notion of Gromov--Hausdorff distance to compare compact metric spaces.  
When such a convergence holds, the continuous limit  highlights some properties of the discrete objects that approximate it, and vice--versa. 

As a first example, consider $T_n$ a tree picked uniformly at random in the set of trees with $n$ vertices labelled by $\{1,\ldots,n\}$. The tree $T_n$ has to be understood as a \emph{typical element} of this set of trees. In this case the answer to the previous question dates back to a series of works by Aldous in the beginning of the 90's \cite{Ald91a, Ald91, Ald93}: Aldous showed that
\begin{equation}
\label{IntroAldous}
\frac{T_n}{2\sqrt n}  \underset{n \rightarrow \infty}{\overset{\mathrm{(d)}}\longrightarrow} \mathcal T_{\mathrm{Br}}
\end{equation}
where the limiting tree is called the Brownian Continuum Random Tree (CRT), and can be constructed from a standard Brownian excursion. This result has various interesting consequences, e.g. its gives the asymptotics in distribution of the diameter, the height (if we consider rooted versions of the trees) and several other statistics related to the tree $T_n$. Consequently it also gives the asymptotic proportion of trees with $n$ labelled vertices that have a diameter larger than $x\sqrt n \ $ or/and a height larger than $y\sqrt n$, etc. Some of these questions on statistics of uniform trees were already treated in previous works,  the strength of Aldous' result is that it describes the asymptotic of the \emph{whole} tree $T_n$. 

Aldous has actually established a version of the convergence (\ref{IntroAldous}) in a much broader context, that of \emph{conditioned Galton--Watson trees} with finite variance. In this situation, $T_n$ is the genealogical tree of a Galton--Watson process (with a given, fixed offspring distribution with mean one and finite variance) conditioned on having a total number of vertices equal to $n,n\geq 1$. Multiplied by $1/\sqrt n$, this tree  converges in distribution to  the Brownian CRT multiplied by a constant that only depends on the variance of the offspring distribution. This should be compared with (and is related to) the convergence of rescaled sums of i.i.d. random variables towards the normal distribution and  its functional analog, the convergence of rescaled random walks towards the Brownian motion. It turns out that the above sequence of uniform labelled trees can be seen as  a sequence of conditioned Galton--Watson trees (when the offspring distribution is a Poisson distribution) and more generally that several sequences of \emph{combinatorial} trees reduce to conditioned Galton--Watson trees. In the early 2000s, Duquesne \cite{Duq03} extended Aldous's result to conditioned Galton--Watson trees with offspring distributions in the domain of attraction of a stable law. We also refer to \cite{DLG02, KortSimple} for related results. In most of these cases the scaling sequences $(a_n)$ are asymptotically much smaller, i.e. $a_n\ll \sqrt n$, and other continuous trees arise in the limit, the so--called family of \emph{stable Lévy trees}.
All these results on conditioned Galton--Watson trees are now well--established, and have a lot of applications in the study of large random graphs (see e.g. Miermont's book \cite{MiermontStFlour} for the connections with random maps and Addario--Berry et al. \cite{ABBG12} for connections with Erd\H{o}s--R\'enyi random graphs in the critical window).

The classical proofs to establish the scaling limits of Galton--Watson trees rely on a careful study of their so--called \emph{contour functions}. It is indeed a common approach to encode trees into functions (similarly to the encoding of the Brownian tree by the Brownian excursion), which are more familiar objects. It turns out that for Galton--Watson trees, the contour functions are closely related to random walks, whose scaling limits are well--known. Let us also mention that another common approach to study large random combinatorial structures is to use technics of analytic combinatorics, see \cite{FlS09} for a complete overview of the topic. None of these two methods will be used here.

In these lectures, we will focus on another point of view, that of sequences of random trees that satisfy a certain \emph{Markov--Branching property}, which appears naturally in a large set of models and includes conditioned Galton--Watson trees. This property is a sort of discrete fragmentation property which roughly  says  that in each tree of the sequence, the  subtrees  above a  given height are independent with a law that depends only on their total size. Under appropriate assumptions, we will see that Markov--Branching sequences of trees, suitably rescaled, converge to  a family of continuous fractal trees, called the \emph{self--similar fragmentation trees}.  These continuous trees  are related to the self--similar fragmentation processes studied by Bertoin in the 2000s \cite{BertoinBook}, which are models used to describe the evolution of objects that randomly split as times passes. The main results on Markov--Branching trees presented here were developed in the paper \cite{HM12}, which has its roots in the earlier paper \cite{HMPW08}.  Several applications have been developed in these two papers, and in more recent works \cite{BerFire,HS15,Riz15}:  to Galton--Watson trees with arbitrary degree constraints, to several combinatorial trees families, including the P\'olya trees (i.e. trees uniformly distributed in the set of rooted, unlabelled, unordered trees with $n$ vertices, $n\geq 1$), to several examples of dynamical models of tree growth and to sequence of \emph{cut--trees}, which describe the genealogy of some deletion procedure of edges in trees.  The objective of these notes is to survey and gather these results, as well as further related results.

In Section \ref{sec:def} below, we will start with a series of definitions related to discrete trees and then present several classical examples of sequences of random trees. We will also introduce there the Markov--Branching property. In Section \ref{SectionGW} we set up the topological framework in which we will work, by introducing the notions of real trees and Gromov--Hausdorff topology. We also recall there the classical results of Aldous \cite{Ald91} and Duquesne \cite{Duq03} on large conditioned Galton--Watson trees. Section \ref{SCMB} is the core of these lectures. We present there the results on scaling limits of Markov--Branching trees, and give the main ideas of the proofs.  The key ingredient is the study of an integer--valued  Markov chain describing the sizes of the subtrees containing a typical leaf of the tree. Section \ref{sec:applic} is devoted to the applications mentioned above. Last, Section \ref{sec:further} concerns  further perspectives and related models (multi--type trees, local limits, applications to other random graphs).

All the sequences of trees we will encounter here have a power growth. There is however a large set of random trees that naturally arise in applications  that do not have such a behavior. In particular, many models of trees arising in the analysis of algorithms have a logarithmic growth. See e.g. Drmota's book \cite{Dr09} for an overview of the most classical models. These examples do not fit into our framework. 

\bigskip

\textbf{Acknowledgements.} I would like to thank the organizers of the \emph{XII Simposio de Probabilidad y Procesos Estocásticos}, held in Mérida, as well as my collaborators on the topic  of Markov--Branching trees, Grégory Miermont, Jim Pitman, Robin Stephenson and Matthias Winkel.

\section{Discrete trees, examples and motivations}
\label{sec:def}

\subsection{Discrete trees}
\label{sec:discrete}

A \textbf{discrete tree} (or graph--theoretic tree) is a finite or countable graph $(V,E)$  that is connected and has no cycle. Here $V$ denotes the set of vertices of the graph and $E$ its set of edges. Note that two vertices are then connected by exactly one path and  that $\# V=\# E+ 1$ when the tree is finite.


In the following, we will often denote a (discrete) tree by the letter $\mathsf t$, and for $\mathsf t=(V,E)$ we will use the slight abuse of notation $v \in \mathsf t$ to mean $v \in V$.

A  tree $\mathsf t$ can be seen as a metric space, when endowed with the \textbf{graph distance} $d_{\mathrm{gr}}$: given two vertices $u,v \in \mathsf t$,  $d_{\mathrm{gr}}(u,v)$ is defined as the number of edges of the path from $u$ to $v$.


A \textbf{rooted tree} $(\mathsf t,\rho)$ is an ordered pair where $\mathsf t$ is a tree and $\rho \in \mathsf t$.  The vertex $\rho$ is then called the {root} of $\mathsf t$.  This gives a genealogical structure to the tree. The root corresponds to the generation 0, its neighbors can be interpreted as its children and form the generation 1, the children of its children form the generation 2, etc. We will usually call the \textbf{height} of a vertex its generation, and denote it by $\mathsf{ht}(v)$ (the height of a vertex is therefore its distance to the root). The height of the tree is then
$$
\mathsf{ht}(\mathsf t)=\sup_{v \in \mathsf t} \mathsf{ht}(v)
$$
and its \textbf{diameter}
$$
\mathsf{diam}(\mathsf t)=\sup_{u,v \in \mathsf t} d_{\mathrm{gr}}(u,v).
$$

The \textbf{degree of a vertex $v \in \mathsf t$} is the number of connected components obtained when removing $v$ (in other words, it is the number of neighbors of $v$). A vertex $v$ different from the root and of degree 1 is called a \textbf{leaf}. In a rooted tree, the \textbf{out--degree of a vertex $v$} is the number of children of $v$. A (full) \textbf{binary} tree is a rooted tree where  all vertices but the leaves have out--degree 2. A \textbf{branch--point} is a vertex of degree at least 3. 


In these lectures, we will mainly work with  rooted trees.  Moreover we will consider, unless specifically mentioned, that two  isomorphic trees are equal, or, when the trees are rooted, that  \textbf{two root--preserving isomorphic trees are equal}. Such trees can be considered as \textit{unordered unlabelled} trees, in opposition to the following definitions.

\bigskip

\textbf{Ordered or/and labelled trees.} In the context of rooted trees, it may happen that one needs to order the children of the root, and then, recursively, the children of each vertex in the tree. This gives an \textbf{ordered} (or planar) tree.
Formally, we generally see such a tree as a subset of the infinite Ulam--Harris tree
$$
\mathcal U=\bigcup_{n=0}^{\infty} \mathbb N^n
$$ 
where $\mathbb N:=\{1,2,\ldots\}$ and $\mathbb N^{0}=\{\varnothing\}$. The element $\varnothing$ is the root of the Ulam--Harris tree, and all other $u=u_1u_2\ldots u_n \in \mathcal U\backslash \{\varnothing\}$ is connected to the root via the unique shortest path 
$$
\varnothing \rightarrow u_1 \rightarrow u_1u_2 \rightarrow \ldots \rightarrow u_1\ldots u_n.
$$
The height (or generation) of such a sequence $u$ is therefore its length, $n$. We then say that $\mathsf t \subset \mathcal U$ is  a (finite or infinite) rooted \emph{ordered} tree  if:
\begin{enumerate}
\item[$\bullet$] $\varnothing \in \mathsf t$ 
\item[$\bullet$] if $u=u_1\ldots u_n \in \mathsf t\backslash \{\varnothing\}$, then $u=u_1\ldots u_{n-1} \in \mathsf t$ (the parent of an individual in $\mathsf t$ that is not the root  is also in $\mathsf t$) 
\item[$\bullet$]  if $u=u_1\ldots u_n \in \mathsf t$, there exists an integer $c_u(\mathsf t)\geq 0$ such that the element $u_1\ldots u_n j \in \mathsf t$ if and only if $1 \leq j \leq c_u(\mathsf t)$.
\end{enumerate}
The number $c_u(\mathsf t)$ corresponds the number of children of $u$ in $\mathsf t$.


We will also sometimes consider \textbf{labelled} trees. In these cases, the vertices are labelled, typically by $\{1,\ldots,n\}$ if there are $n$ vertices (whereas in an unlabelled tree, the vertices but the root are indistinguishable). Partial labelling is also possible, e.g. by labelling only the leaves of the tree.

\textbf{In the following we will always specified when a tree is ordered or/and labelled. When not specified, it is implicitly unlabelled, unordered.}

\bigskip

\textbf{Counting trees.} It is sometimes possible, but not always, to have explicit formul{\ae} for the number of trees of a specific structure. For example, it is known that the number of trees with $n$ labelled vertices is
$$
n^{n-2} \quad \text{(Cayley formula)},
$$
and consequently, the number of rooted trees with $n$ labelled vertices is 
$$
n^{n-1}.
$$
The number of rooted ordered binary trees with $n+1$ leaves is
$$
\frac{1}{n+1} \binom{2n}{n} 
$$
(this number is called the $n$th Catalan number)
and the number of rooted ordered with with $n$ vertices is 
$$
\frac{1}{n} \binom{2n-2}{n-1}.
$$
On the other hand, there is no explicit formula for the number of rooted (unlabelled, unordered) trees.  Otter \cite{Otter48} shows that this number is asymptotically proportional to
$$
c \kappa^n n^{-3/2}
$$
where $c \sim 0.4399$ and $\kappa \sim 2.9557$. This should be compared to the asymptotic expansion of the $n$th Catalan number, which is proportional (by Stirling's formula) to $\pi^{-1/2} 4^n n^{-3/2}$.

We refer to the book of Drmota \cite{Dr09} for more details and technics, essentially based on generating functions. 

\subsection{First examples}
\label{sec:examples}

We now present a first series of classical families of random trees. Our goal will be to describe their scaling limits when the sizes of the trees grow, as discussed in the Introduction. This will be done in Section \ref{sec:applic}. Most of these families (but not all) share a common property, the Markov--Branching property that will be introduced in the next section.

\bigskip

\textbf{Combinatorial trees.} Let $\mathbb T_n$ denote a finite set of trees with $n$ vertices, all sharing some structural properties. E.g. $\mathbb T_n$ may 
be the set of all rooted trees with $n$ vertices, or the set of all  rooted ordered trees with $n$ vertices, or the set of all  binary trees with $n$ vertices, etc. We are interested in the asymptotic behavior of a ``typical element" of $\mathbb T_n$ as $n \rightarrow \infty$. That is, we pick a tree \emph{uniformly} at random in $\mathbb T_n$, denote it by $T_n$ and study its scaling limit. The  global behavior of $T_n$ as $n \rightarrow \infty$ will represent some of the features shared by most of the trees. For example, if the probability that the height of $T_n$ is larger than $n^{\frac{1}{2}+\varepsilon}$ tends to 0 as $n \rightarrow \infty$, this means that the proportion of trees in the set that have a height larger than $n^{\frac{1}{2}+\varepsilon}$ is asymptotically negligeable, etc. We will more specifically be interested in the following cases: 
\begin{enumerate}
\item[$\bullet$] $T_n$ is a uniform rooted tree with $n$ vertices 
\item[$\bullet$] $T_n$ is a uniform rooted ordered tree with $n$ vertices
\item[$\bullet$] $T_n$ is a uniform  tree with $n$ labelled vertices 
\item[$\bullet$] $T_n$ is a uniform rooted ordered  binary tree with $n$ vertices ($n$ odd)
\item[$\bullet$] $T_n$ is a uniform rooted binary tree with $n$ vertices ($n$ odd),
\end{enumerate}
etc. Many variations are of course possible, in particular one may consider trees picked uniformly amongst sets of trees with a given structure and $n$ \emph{leaves}, or more general degree constraints. Some of these uniform trees will appear again in the next example.


\bigskip

\textbf{Galton--Watson trees}. Galton--Watson trees are random trees describing the genealogical structure of Galton--Watson processes. These are simple mathematical models for the evolution of a population that continue to play un important role in probability theory and in applications.  Let $\eta$ be a probability on $\mathbb Z_+$ -- $\eta$ is called the \emph{offspring distribution} --  and let $m:=\sum_{i \geq 1}i \eta(i) \in [0,\infty]$ denote its mean. Informally, in a Galton--Watson tree with offspring distribution $\eta$, each vertex has a random number of children distributed according to $\eta$, independently. 
We will always assume that $\eta(1)<1$ in order to avoid the trivial case where each individual has a unique child. Formally, a $\eta$--Galton--Watson tree $T^{\eta}$ is usually seen as an ordered rooted tree and defined as follows (recall the Ulam--Harris notation $\mathcal U$): 
\begin{enumerate}
\item[$\bullet$] $c_{\varnothing}(T^{\eta})$ is distributed according to $\eta$ 
\item[$\bullet$] conditionally on $c_{\varnothing}(T^{\eta})=p$, the $p$ ordered subtrees $\tau_i=\{u \in \mathcal U:iu \in T^{\eta}\}$ descending from $i=1,\ldots,p$ are independent and distributed as $T^{\eta}$.
\end{enumerate}
From this construction, one  sees that the distribution of $T^{\eta}$ is given by:
\begin{equation}
\label{lawGW}
\mathbb P\left(T^{\eta}=\mathsf t\right)=\prod_{v \in \mathsf t} \eta_{c_v(\mathsf t)}
\end{equation}
 for all rooted ordered tree $\mathsf t$.
This definition  of Galton--Watson trees as ordered trees is the simplest, avoiding any symmetry problems. However in the following we will mainly see these trees up to isomorphism, which roughly means that we can ``forget the order''.

Clearly, if we call $Z_k$ the number of individuals at height $k$, then $(Z_k,k\geq 1)$ is a Galton--Watson process starting from $Z_0=1$. It is well-known that the extinction time of this process,
$$
\inf\{k \geq 0 :Z_k=0\}
$$
if finite with probability 1 when $m \leq 1$ and with a probability $\in [0,1)$ when $m>1$. The offspring distribution $\eta$ and the tree $T^{\eta}$ are said to be \emph{subcritical} when $m<1$,  \emph{critical} when $m=1$ and  \emph{supercritical} when $m>1$. From now on, we assume that
$$
m=1
$$
and for integers $n$ such that $\mathbb P(\# T^{\eta}=n)>0$, we let $T_n^{\eta,\mathrm v}$ denote a \emph{non--ordered} version of the Galton--Watson tree $T^{\eta}$ conditioned to have $n$ vertices. Sometimes, we will need to keep the order and we will let $T_n^{\eta,\mathrm v,\mathrm{ord}}$ denote this ordered conditioned version. We point out that in most cases, \emph{but not all}, a subcritical or a supercritical Galton--Watson tree conditioned to have $n$ vertices is distributed as a critical  Galton--Watson tree conditioned to have $n$ vertices with a different offspring distribution. So the assumption $m=1$ is not too restrictive. We refer to \cite{Janson12} for details on that point.

It turns out that conditioned Galton--Watson trees are closely related to combinatorial trees. Indeed, one can easily check with (\ref{lawGW}) that:
\begin{enumerate}
\item[$\bullet$] if $\eta \sim$ Geo(1/2), $T^{\eta,\mathrm v,\mathrm{ord}}_n$ is uniform amongst the set of rooted ordered trees with $n$ vertices
\item[$\bullet$] if $\eta \sim $ Poisson(1), $T^{\eta, \mathrm v}_n$ is uniform amongst the set of rooted trees with $n$ labelled vertices
\item[$\bullet$] if $\eta \sim \frac{1}{2}\left(\delta_0 + \delta_2 \right)$, $T^{\eta,\mathrm v,\mathrm{ord}}_n$ is uniform amongst the set of rooted ordered binary trees with $n$ vertices.
\end{enumerate}
We refer e.g. to Aldous \cite{Ald91} for additional examples. 

Hence, studying the large--scale structure of conditioned Galton--Watson trees will also lead to results in the context of combinatorial trees. As mentioned in the Introduction, the scaling limits of large conditioned Galton--Watson trees are now  well--known. Their study has been initiated by Aldous \cite{Ald91a,Ald91,Ald93} and then expanded by Duquesne \cite{Duq03}.  This will be reviewed in Section \ref{SectionGW}. However, there are some sequences of combinatorial trees that \emph{cannot} be reinterpreted as Galton--Watson trees, starting  with the example of the uniform rooted tree with $n$ vertices or the uniform rooted binary tree with $n$ vertices. Studying the scaling limits of these tree remained open for a while, because of the absence of symmetry properties. There scaling limits are presented in Section \ref{sec:Polya}.

In another direction, one may also wonder what happens when considering versions of  Galton--Watson trees conditioned to have $n$ leaves, instead of $n$ vertices, or more general degree constraints. This is discussed in Section \ref{GWmore}.

\bigskip

\textbf{Dynamical models of tree growth}. We now turn to several sequences of finite rooted random trees  that are built recursively by adding at each step new edges on the pre--existing tree. We start with a well--known algorithm that Rémy \cite{Rem85} introduced to generate uniform binary trees with $n$ leaves.

\emph{\textbf{Rémy's algorithm.}} The sequence $(T_n(\mathrm R),n\geq 1)$ is constructed recursively as follows:
\begin{enumerate}
\item[$\bullet$] Step $1$: $T_1(\mathrm R)$ is the tree with one edge and two vertices: one root, one leaf
\item[$\bullet$] Step $n$: given $T_{n-1}(\mathrm R)$, choose uniformly at random one of its edges and graft on ``its middle" one new edge-leaf, that is split the selected edge into two so as to obtain two edges separated by a new vertex, and then add a new edge-leaf to the new vertex. This gives $T_n(\mathrm R)$.
\end{enumerate}
It turns out (see e.g. \cite{Mar03}) that the tree $T_n(\mathrm R)$, to which has been subtracted the edge between the root and the first branch point, is distributed as a binary critical Galton--Watson tree conditioned to have $2n-1$ vertices, or equivalently $n$ leaves (after forgetting the order in the GW--tree).  As so, we deduce its asymptotic behavior from that of Galton--Watson trees. However  this model can be extended in several directions, most of which are not related to Galton--Watson trees. We detail three of them.  

\emph{\textbf{Ford's $\alpha$-model}} \cite{Ford05}. Let $\alpha \in [0,1]$. We construct a sequence $(T_n(\alpha),n\geq 1)$ by modifying Rémy's algorithm as follows: 
\begin{enumerate}
\item[$\bullet$] Step $1$: $T_1(\alpha)$ is the tree with one edge and two vertices: one root, one leaf
\item[$\bullet$] Step $n$: given $T_{n-1}(\alpha)$, give a weight $1-\alpha$ to each edge connected to a leaf, and $\alpha$ to all other edges (the internal edges). The total weight is $n-\alpha$. Now choose an edge at random with a probability proportional to its weight and graft on ``its middle" one new edge-leaf. This gives $T_n(\alpha)$.
\end{enumerate}
Note that when $\alpha=1/2$ the weights are the same on all edges and we recover Rémy's algorithm. When $\alpha=0$, the new edge is always grafted uniformly on an edge--leaf, which gives a tree $T_n(0)$ known as the \emph{Yule tree} with $n$ leaves. When $\alpha=1$, we obtain a deterministic tree called the \emph{comb tree}. This family of trees indexed by $\alpha\in [0,1]$ was introduced by Ford \cite{Ford05} in order to interpolate between the Yule, the uniform and the comb models. His goal was to propose  new models for phylogenetic trees.

\emph{\textbf{$k$--ary growing trees}} \cite{HS15}. This is another extension of Rémy's algorithm, where now several edges are added at each step. Consider an integer $k\geq 2$.  The sequence $(T_n(k),n\geq 1)$ is constructed recursively as follows:
\begin{enumerate}
\item[$\bullet$] Step $1$: $T_1(k)$ is the tree with one edge and two vertices: one root, one leaf
\item[$\bullet$] Step $n$: given $T_{n-1}(k)$, choose uniformly at random one of its edges and graft on ``its middle" $k-1$ new edges--leaf. This gives $T_n(k)$.
\end{enumerate}
When $k=2$, we recover Rémy's algorithm. For larger $k$, there is no connection with Galton--Watson trees.

\emph{\textbf{Marginals of stable trees -- Marchal's algorithm}}. In \cite{Mar06}, Marchal considered the following algorithm,  that attributes weights also to the vertices. Fix a parameter $\beta \in (1,2]$ and 
construct the sequence $(T_n(\beta),n\geq 1)$ as follows:
\begin{enumerate}
\item[$\bullet$] Step $1$: $T_1(\beta)$ is the tree with one edge and two vertices: one root, one leaf
\item[$\bullet$] Step $n$: given $T_{n-1}(\beta)$, attribute the weight
\begin{enumerate}
\item[$\circ$] $\beta-1$ on each edge
\item[$\circ$] $d-1-\beta$ on each vertex of degree $d\geq 3$.
\end{enumerate} 
The total weight is $n\beta-1$.
Then select at random an edge or vertex with a probability proportional to its weight and graft on it a new edge-leaf. This gives $T_n(\beta)$. 
\end{enumerate}
The reason why Marchal introduced this algorithm is that $T_n(\beta)$ is actually distributed as the shape of a tree with edge--lengths that is obtained by sampling $n$ leaves at random in the stable Lévy tree with index $\beta$.  The class of stable Lévy trees plays in important role in the theory of random trees. It is introduced in Section \ref{SCGW} below.

Note that when $\beta=2$, vertices of degree $3$ are never selected (their weight is 0). So the trees $T_n(\beta),n\geq 1$ are all binary, and we recover Rémy's algorithm.

\bigskip

Of course, several other extensions of trees built  by adding edges recursively may be considered, some of which are mentioned in Section \ref{sec:marginals} and Section \ref{sec:multitype}.

\textbf{Remark.} In these dynamical models of tree growth, we build \emph{on a same probability space} the sequence of trees, contrary to the examples of Galton--Watson trees or combinatorial trees that give sequences of \emph{distributions} of trees.  In this situation, one may expect to have more than a convergence in distribution for the rescaled sequences of trees. We will see in Section \ref{sec:cvdynam} that it is indeed the case.

\subsection{The Markov--Branching property} 
\label{sec:MB}

Markov--Branching trees were introduced by Aldous \cite{Ald96} as a class of random binary trees for  phylogenetic models and later extended to non--binary cases in Broutin \& al. \cite{BDMcLDlS08}, and  Haas \& al. \cite{HMPW08}. It turns out that many natural models of sequence of trees  satisfy the \textbf{Markov--Branching property} (\textbf{MB--property} for short), starting with the example of conditioned Galton--Watson trees and most of the examples of the previous section.

Consider
$$
\big(T_n,n\geq 1 \big)
$$
a sequence of trees where $T_n$ is a rooted (unordered, unlabelled) tree with $n$ \emph{leaves}. The MB--property is a property of the sequence of \emph{distributions} of $T_n, n \geq 1$. Informally, the MB--property says that for each tree $T_n$, \emph{given that} 
\begin{center}
``the root of $T_n$ splits in $p$ subtrees with respectively $\lambda_1 \geq \ldots \geq \lambda_p$ leaves",
\end{center}
 then $T_n$ is distributed as the tree obtained by gluing on a common root $p$ independent trees with respective distributions those of $T_{\lambda_1},\ldots,T_{\lambda_p}$. The way the leaves are distributed in the sub--trees above the root, in each $T_n$, for $n\geq 1$, will then allow to fully describe the distributions of the $T_n,n \geq 1$.
 
We now explain rigorously how to build such  sequences of trees. We start with a sequence of probabilities $(q_n,n\geq 1),$ where for each $n$,
$q_n$ is a probability on the set of partitions of the integer $n$. If $n \geq 2$, this set is defined by
$$
\mathcal P_n:=\left\{\lambda=(\lambda_1,\ldots,\lambda_p), \lambda_i\in \mathbb N, \lambda_1 \geq \ldots \geq \lambda_p\geq 1 : \sum_{i=1}^p \lambda_i=n \right\},
$$
whereas if  $n=1$, $\mathcal P_1:=\{(1),\emptyset\}$  (we need to have a cemetery point).
For a partition $\lambda \in \mathcal P_n$, we denote by $p(\lambda)$ its length, i.e. the number of terms in the sequence $\lambda$.  The probability $q_n$ will determine how the $n$ leaves of $T_n$ are distributed into the subtrees above its root.  We call such a probability a \emph{splitting distribution}. 
In order that effective splittings occur, we will always assume that
$$
q_n((n))<1, \quad \forall n \geq 1.
$$
We need  to define a notion of \emph{gluing} of trees. Consider $\mathrm t_1,\ldots,\mathrm t_p$, $p$ discrete rooted (unordered) trees. Informally, we want to glue them on a same common root in order to form a tree  $
\langle \mathrm t_1,\ldots,\mathrm t_p \rangle
$ whose root splits into the $p$ subtrees $\mathrm t_1,\ldots,\mathrm t_p$. Formally, this can e.g. be done as follows. Consider first ordered versions of the trees $\mathrm t^{\mathrm{ord}}_1,\ldots,\mathrm t^{\mathrm{ord}}_p$ seen as subsets of the Ulam--Harris tree $\mathcal U$ and then define a new ordered tree by
$$
\langle \mathrm t^{\mathrm{ord}}_1,\ldots,\mathrm t^{\mathrm{ord}}_p \rangle:=\{\varnothing\}\cup_{i=1}^p i \mathrm t^{\mathrm{ord}}_i.
$$
The tree $
\langle \mathrm t_1,\ldots,\mathrm t_p \rangle
$ is then defined as the unordered version of $
\langle \mathrm t^{\mathrm{ord}}_1,\ldots,\mathrm t^{\mathrm{ord}}_p \rangle
$.

\medskip

\begin{defn}
For each $n\geq 1$, let $q_n$ be a probability on $\mathcal P_n$ such that $q_n((n))<1$. From the sequence $(q_n,n \geq 1)$ we construct recursively a sequence of distributions  $(\mathcal L^{\mathbf q}_n)$ such that for all $n\geq 1$, $\mathcal L^{\mathbf q}_n$ is carried by  the set of rooted trees with $n$ leaves,  as follows:
\begin{enumerate}
\item[$\bullet$] $\mathcal L^{\mathbf q}_1$ is the distribution of a line--tree with $G+1$ vertices and $G$ edges where $G$ is a geometric distribution:
$$
\mathbb P(G=k)=q_1(\emptyset)(1-q_1(\emptyset))^k, \quad k\geq 0,
$$ 
\item[$\bullet$] for $n \geq 2$, $\mathcal L^{\mathbf q}_n$ is the distribution of
$$
\langle T_1,\ldots,T_{p(\Lambda)} \rangle
$$
where $\Lambda$ is a partition of $n$ distributed according to $q_n$, and given $\Lambda$, $T_1,\ldots,T_{p(\Lambda)}$ are independent trees with respective distributions $\mathcal L^{\mathbf q}_{\Lambda_1},\ldots,$ $\mathcal L^{\mathbf q}_{\Lambda_{p(\Lambda)}}$.
\end{enumerate}
A sequence $(T_n,n\geq 1)$ of random rooted trees such that $T_n \sim \mathcal L^{\mathbf q}_n$ for each $n \in \mathbb N$  
 is called a \emph{\textbf{MB--sequence of trees \emph{indexed by the leaves}, with splitting distributions $\boldsymbol{(q_n,n\geq 1)}$}}.  
\end{defn}

This construction may be re--interpreted as follows: we start from a collection of $n$ \emph{indistinguishable}
balls, and with probability $q_n(\lambda_1,\ldots,\lambda_p)$, split
the collection into $p$ sub--collections with
$\lambda_1,\ldots,\lambda_p$ balls. Note that there is a chance
$q_n((n))<1$ that the collection remains unchanged during this step of
the procedure. Then, re-iterate the splitting operation independently
for each sub-collection using this time the probability distributions
$q_{\lambda_1},\ldots,q_{\lambda_p}$.  If a sub-collection consists of
a single ball, it can remain single with probability $q_1((1))$ or get
wiped out with probability $q_1(\emptyset)$. We continue the procedure
until all the balls are wiped out. The tree $T_n$ is then the genealogical tree
associated with this process: it is  rooted at the initial collection of $n$ balls and its $n$ leaves
correspond to the $n$ isolated balls just before they are wiped out,
See Figure
\ref{fig:CMB} for an illustration. 

\begin{figure}
\begin{center}
\includegraphics[height=4.5cm]{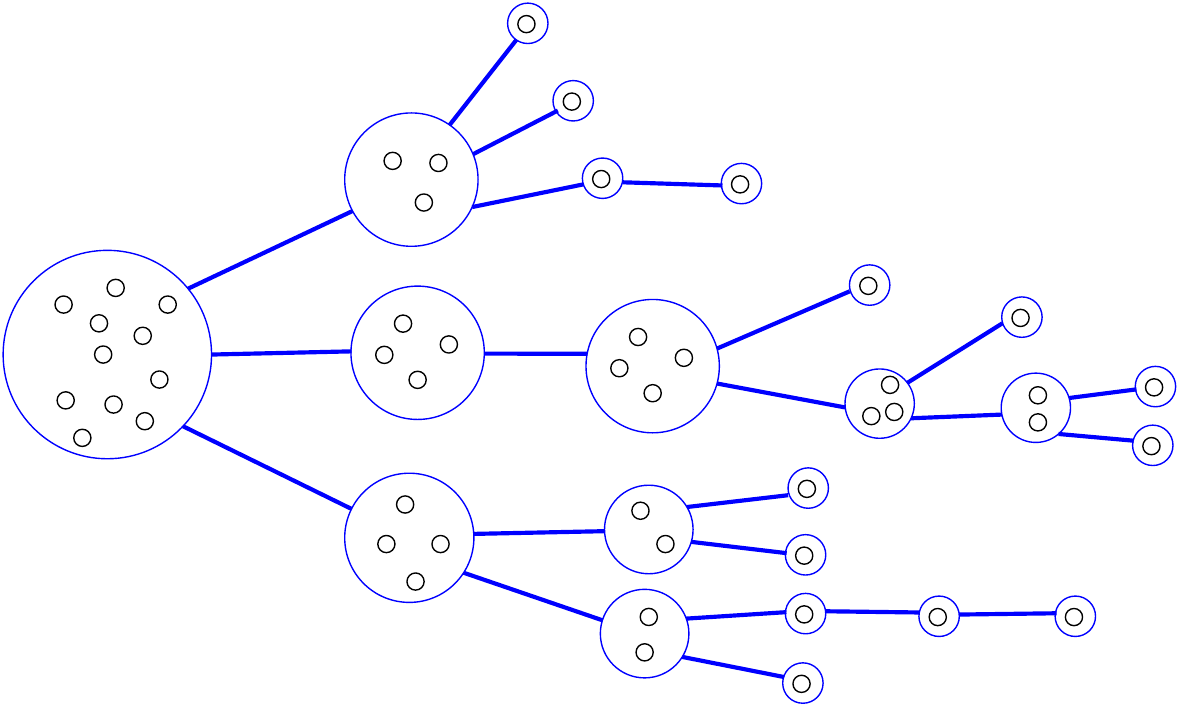}
\end{center}
\caption{A sample tree $T_{11}$. The first splitting arises with
  probability $q_{11}(4,4,3)$.}
\label{fig:CMB}
\end{figure}

We can define similarly MB--sequences of (distributions) of trees indexed by their number of \emph{vertices}. Consider here a sequence $(p_n,n\geq 1)$ such that $p_n$ is a probability on $\mathcal P_n$ with no restriction but 
$$p_1((1))=1.$$   
Mimicking the previous balls in urns construction and starting from a collection of
$n$ indistinguishable balls, we first remove a ball, split the $n-1$ remaining balls in
sub-collections with $\lambda_1,\ldots,\lambda_p$ balls with
probability $p_{n-1}((\lambda_1,\ldots,\lambda_p))$, and iterate
independently on sub-collections until no ball remains. Formally, this gives:

\begin{defn}
For each $n\geq 1$, let $p_n$ be a probability on $\mathcal P_n$,  such that $p_1((1))=1$.
From the sequence $(p_n,n \geq 1)$ we construct recursively a sequence of distributions  $(\mathcal V^{\mathbf p}_n)$ such that for all $n\geq 1$, $\mathcal V^{\mathbf p}_n$ is carried by the set of trees with $n$ vertices, as follows:
\begin{enumerate}
\item[$\bullet$] $\mathcal V^{\mathbf p}_1$ is the deterministic distribution of the tree reduced to one vertex,
\item[$\bullet$] for $n \geq 2$, $\mathcal V^{\mathbf p}_n$ is the distribution of
$$
\langle T_1,\ldots,T_{p(\Lambda)} \rangle
$$
where $\Lambda$ is a partition of $n-1$ distributed according to $p_{n-1}$, and given $\Lambda$, $T_1,\ldots,T_{p(\Lambda)}$ are independent trees with respective distributions $\mathcal V^{\mathbf p}_{\Lambda_1},\ldots, \mathcal V^{\mathbf p}_{\Lambda_{p(\Lambda)}}$.
\end{enumerate}
A sequence $(T_n,n\geq 1)$ of random rooted trees such that $T_n \sim \mathcal V^{\mathbf p}_n$ for each $n \in \mathbb N$  
 is called a \emph{\textbf{MB--sequence of trees indexed by the vertices, with splitting distributions }$\boldsymbol{(p_n,n\geq 1)}$}.  
\end{defn}

More generally, the MB--property can be extended to sequences of trees $(T_n,n\geq 1)$ with arbitrary degree constraints, i.e. such that for all $n$, $T_n$ has $n$ vertices in $A$, where $A$ is a given subset of $\mathbb Z_+$. We will not develop this here and refer the interested reader  to \cite{Riz15} for more details.

\bigskip

\textbf{Some examples.} 
\textbf{1. A deterministic example.}  Consider the splitting distributions on $\mathcal P_n$ 
$$q_n(\lceil n/2\rceil,\lfloor n/2\rfloor)=1,  \quad n \geq 2,$$ 
as well as $q_1(\emptyset)=1$. Let $(T_n,n\geq 1)$ the corresponding MB--sequence indexed by leaves. Then $T_n$ is a deterministic discrete binary tree, whose root splits in two subtrees with both $n/2$ leaves when $n$ is even, and respectively $(n+1)/2$, $(n-1)/2$ leaves when $n$ is odd. Clearly, when $n=2^k$,  the height of $T_n$ is exactly $k$, and more generally for large $n$, $\mathrm{ht}(T_n)\sim \ln(n)/\ln(2)$. 

\medskip

\textbf{2. A basic example.} For $n \geq 2$, let $q_n$ be the probability on $\mathcal P_n$ defined by
$$
q_n((n))=1-\frac{1}{n^{\alpha}} \quad \text{and} \quad q_n(\lceil n/2\rceil,\lfloor n/2\rfloor)=\frac{1}{n^{\alpha}} \quad \text{for some }\alpha>0,
$$
and let $q_1(\emptyset)=1$. Let $(T_n,n\geq 1)$ be a MB--sequence indexed by leaves with splitting distributions $(q_n)$. Then $T_n$ is a discrete tree with vertices with degrees $\in \{1,2,3\}$ where the distance between the root and the first branch point (i.e. the first vertex of degree 3) is a Geometric distribution on $\mathbb Z_+$ with success parameter $n^{-\alpha}$. The two subtrees above this branch point are  independent subtrees, independent of the Geometric r.v. just mentioned, and whose respective distances between the root and first branch point are  Geometric  distributions with respectively $(\lceil n/2\rceil)^{-\alpha}$ and $(\lfloor n/2\rfloor)^{-\alpha}$ parameters. Noticing the weak convergence
$$
\frac{\mathrm{Geo}(n^{-\alpha})}{n^{\alpha}} \overset{\mathrm{(d)}}{\underset{n \rightarrow \infty}\longrightarrow} \mathrm{Exp}(1)
$$
one may expect that $n^{-\alpha}T_n$ has a limit in distribution. We will later see that it is indeed the case.

\medskip

\textbf{3. Conditioned Galton--Watson trees}. Let $T_n^{\eta, \mathrm l}$ be a Galton--Watson tree with offspring distribution $\eta$, conditioned on having $n$ leaves, for integers $n$ for which this is possible. The branching property is then preserved by conditioning and the sequence $(T^{\eta,\mathrm l}_n,n : \mathbb P(\#_{\mathrm{leaves}} T^{\eta})>0)$ is Markov--Branching, with splitting distributions
$$
q^{\mathrm{GW},\eta}_n(\lambda)=\eta(p)\times \frac{p!}{\prod_{i=1}^p m_i(\lambda)!}\times \frac{\prod_{i=1}^{p}\mathbb P(\#_{\mathrm{leaves}} T^{\eta}=\lambda_i)}{\mathbb P(\#_{\mathrm{leaves}}  T^{\eta}=n)}
$$
for all $\lambda \in \mathcal P_n,n\geq 2$, where $\#_{\mathrm{leaves}} T^{\eta}$ is the number of leaves of the unconditioned Galton--Watson tree $T^{\eta}$, and $m_i(\lambda)=\# \{1\leq j \leq p:\lambda_j=i\}$. The probability $q^{\mathrm{GW},\eta}_1$ is given by $q^{\mathrm{GW},\eta}_1((1))=\eta(1)$.

Similarly, if $T_n^{\eta,\mathrm v}$  denotes  a Galton--Watson tree with offspring distribution $\eta$, conditioned on having $n$ \emph{vertices}, the sequence $(T^{\eta,\mathrm v}_n, \mathbb P(\#_{\mathrm{vertices}} T^{\eta})>0)$ is MB, with splitting distributions
\begin{equation}
\label{splitGW}
p^{\mathrm{GW},\eta}_{n-1}(\lambda)=\eta(p)\times \frac{p!}{\prod_{i=1}^p m_i(\lambda)!}\times \frac{\prod_{i=1}^{p}\mathbb P(\#_{\mathrm{vertices}} T^{\eta}=\lambda_i)}{\mathbb P(\#_{\mathrm{vertices}}  T^{\eta}=n)}
\end{equation}
for all $\lambda \in \mathcal P_{n-1},n \geq 3$ where $\#_{\mathrm{vertices}} T^{\eta}$ is the number of leaves of the unconditioned GW--tree $T^{\eta}$. Details can be found in \cite[Section 5]{HM12}.

\medskip

\textbf{4. Dynamical models of tree growth.}  Rémy's, Ford's, Marchal's and the $k$--ary algorithms all lead to MB--sequences of trees indexed by leaves. To be precise, we have to remove in each of these trees the edge adjacent to the root to obtain MB--sequences of trees (the roots have all a unique child). The MB--property can be proved by induction on $n$.  By construction, the distribution of the leaves in the subtrees above the root is closely connected to urns models.  We have the following expressions for the splitting distributions: 

\bigskip

\emph{{\textbf{Ford's $\alpha$--model.}}} For $k \geq \frac{n}{2}, n\geq 2$,
$$
q_n^{\mathrm{Ford},\alpha}(k,n-k)=\left(1+\mathbbm 1_{k\neq \frac{n}{2}}\right)\frac{\Gamma(k-\alpha)\Gamma(n-k-\alpha)}{\Gamma(n-\alpha)\Gamma(1-\alpha)}\left(\frac{\alpha}{2}\binom{n}{k}+(1-2\alpha) \binom{n-2}{k-1}\right),
$$
and $q_1(\emptyset)=1$. See \cite{Ford05} for details. 
In particular, taking $\alpha=1/2$ one sees that 
$$
q_n^{\text{Rémy}} (k,n-k)=\frac{1}{4}\left(1+\mathbbm 1_{k\neq \frac{n}{2}}\right)\frac{\Gamma(k-1/2)\Gamma(n-k-1/2)}{\Gamma(n-1/2)\Gamma(1-1/2)}\binom{n}{k}, \quad k \geq \frac{n}{2}, n \geq 2.
$$

\bigskip
\emph{{\textbf{$k$-ary growing trees.}}} Note that in these models, there are $1+(k-1)(n-1)$ leaves in the tree $T_n(k)$, so that the indices do not exactly correspond to the definitions of the Markov--Branching  properties seen in the previous section. However,  by relabelling,  defining for $m=1+(k-1)(n-1)$ the tree $\overline T_m(k)$ to be the tree $T_n(k)$ to which the edge adjacent to the root has been removed, we obtain a MB--sequence $(\overline T_m(k), m \in (k-1)\mathbb N+2-k)$. The splitting distributions are defined for $m = 1+(k-1)(n-1), n \geq 2$ and 
 $\lambda=(\lambda_1,\ldots,\lambda_k) \in \mathcal P_m$ such that $\lambda_i=1+(k-1)\ell_i,$ for some $\ell_i \in \mathbb Z_+$ for all $i$ (note that $\sum_{i=1}^k \ell_i=n-2$) by 
$$
q_m^{k}(\lambda)=\sum_{\mathbf n=(n_1,\ldots,n_k)\in \mathbb N^k : \mathbf n^{\downarrow}=\lambda}\overline q_m(\mathbf n)
$$
where $\mathbf n^{\downarrow}$ is the decreasing rearrangement of the elements of $\mathbf n$ and 	
\begin{equation*}
\overline q_m(\mathbf n)=\frac{1}{k(\Gamma(\frac{1}{k}))^{k-1}}\left(\prod_{i=1}^k \frac{\Gamma(\frac{1}{k}+n_i)}{n_i!}\right)\frac{(n-2)!}{\Gamma(\frac{1}{k}+n-1)} \left(\sum_{j=1}^{n_1+1} \frac{n_1!}{(n_1-j+1)!}\frac{(n-j-1)!}{(n-2)!}\right).
\end{equation*}
See \cite[Section 3]{HS15}.

\bigskip
\emph{{\textbf{Marchal's algorithm.}}} For $\lambda=(\lambda_1,\ldots,\lambda_p) \in \mathcal P_n, n \geq 2$,
$$
q_n^{\mathrm{Marchal},\beta}(\lambda)=\frac{n!}{ \lambda_1 ! \ldots \lambda_p ! m_1(\lambda) ! \ldots m_n(\lambda)! }\frac{\beta^{2-p}\Gamma(2-\beta^{-1})\Gamma(p-\beta)}{\Gamma(n-\beta^{-1})\Gamma(2-\beta)}\prod_{i=1}^p \frac{\Gamma(\lambda_j-\beta^{-1})}{\Gamma(1-\beta^{-1})}
$$
where $m_i(\lambda)=\# \{1\leq j \leq p:\lambda_j=i\}$.
This is a consequence of \cite[Theorem 3.2.1]{DLG02} and \cite[Lemma 5]{Mier03}.

\medskip

\textbf{5. Cut--trees.} 
\textit{Cut--tree of a uniform Cayley tree}.  Consider $C_n$ a uniform Cayley tree of size $n$, i.e. a tree picked uniformly at random amongst the set of rooted  tree with $n$ labelled vertices. This tree has the following \emph{recursive property} (see Pitman \cite[Theorem 5]{PitCoal99}): removing an edge uniformly at random in $C_n$ gives two trees, which given their numbers of vertices, $k,n-k$ say, are independent uniform Cayley trees of respective sizes $k,n-k$. Now, consider the following deletion procedure: remove in $C_n$ one edge uniformly at random, then remove another edge in the remaining set of $n-2$ edges uniformly at random and so on until all edges have been removed. It was shown by Janson \cite{Jan06} and Panholzer \cite{Pan06}   that the number of steps needed to isolate the root divided by $\sqrt n$ converges in distribution to a Rayleigh distribution (i.e. with density $x\exp(-x^2/2)$ on $\mathbb R_+$). 
Bertoin \cite{BerFire} was more generally interested in the number of steps needed to isolate $\ell$ distinguished vertices, and in that aim he introduced the \emph{cut--tree} $T^{\mathrm{cut}}_n$ of $C_n$. The tree $T^{\mathrm{cut}}_n$ is the genealogical tree of the above deletion procedure, i.e. it describes the genealogy of the connected components, see Figure \ref{FigBertoin} for an illustration and  \cite{BerFire} for a precise construction of $T^{\mathrm{cut}}_n$. Let us just mention here that $T^{\mathrm{cut}}_n$ is a rooted binary tree with $n$ leaves, and that Pitman's recursive property implies that $(T^{\mathrm{cut}}_n,n\geq 1)$ is MB. The corresponding splitting probabilities are:
$$
q^{\mathrm{Cut, Cayley}}_n(k,n-k)= \frac{(n-k)^{n-k-1}}{(n-k)!}\frac{k^{k-1}}{k!}\frac{(n-2)!}{n^{n-3}}, \quad n/2<k\leq n-1,
$$ 
the calculations are detailed in \cite{BerFire,Pavlov77}. 

\begin{figure}
\begin{center}
\includegraphics[height=11cm,angle=-90]{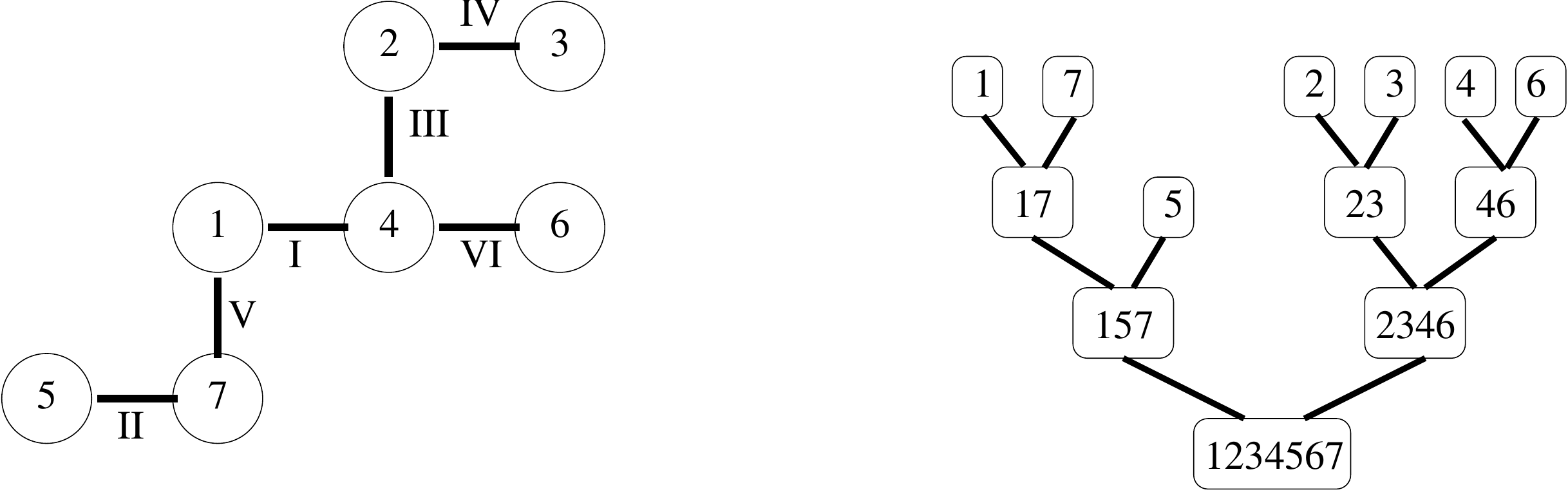}
\end{center}
\caption{On the left, a version of the tree $C_7$, with edges labelled in order of deletion. On the right the associated cut--tree $T^{\mathrm{cut}}_7$, whose vertices are the different connected components arising in the deletion procedure.}
\label{FigBertoin}
\end{figure}

\textit{Cut--tree of a uniform recursive tree}. A recursive tree  with $n$ vertices is a tree with vertices labelled by $1, \ldots, n$, rooted at 1, such that  the sequence of labels of vertices along any branch from the root to a leaf is increasing. It turns out that the cut--tree  of a uniform recursive tree is also MB and with splitting probabilities
$$
q^{\mathrm{Cut, Recursive}}_n(k,n-k)= \frac{n}{(n-1)}\left(\frac{1}{k(k+1)}+\frac{1}{(n-k)(n-k+1)}\right) \quad n/2<k\leq n-1,
$$   
see \cite{BerPercoRTT}.

\bigskip

\textbf{Remark.} The first example is a simple example of models where \emph{macroscopic branchings are frequent}, unlike the second example where  \emph{macroscopic branchings are rare} (they occur with probability $n^{-\alpha} \rightarrow 0$). By macroscopic branchings, we mean that the way that the $n$ leaves (or vertices) are distributed  above the root gives at least two subtrees with a size proportional to $n$. Although it is not completely obvious yet, nearly all other examples above have rare macroscopic branchings (in a sense that will be specified later) and  this is typically the context in which we will study the scaling limits of MB--trees. Typically the tree $T_n$ will then grow as a power of $n$.  When macroscopic branchings are frequent, there is no  scaling limit in general for the Gromov--Hausdorff topology, a topology introduced in the next section. However it is known that the height of the tree $T_n$ is then often of order $c\ln(n)$. This case has been studied in \cite{BDMcLDlS08}.

\section{The example of Galton--Watson trees and topological framework}
\label{SectionGW}

We start with an informal version of the prototype result of Aldous on the description of the scaling limits of conditioned Galton--Watson trees. Let $\eta$ be a critical offspring distribution with \emph{finite variance} $\sigma^2 \in (0,\infty)$, and let $T_n^{\eta,\mathrm v}$ denote a Galton--Watson tree with offspring distribution $\eta$, conditioned to have $n$ vertices (in the following it is implicit that we only consider integers $n$ such that this conditioning is possible). Aldous \cite{Ald93} showed that
\begin{equation}
\label{cvAldous}
\frac{\sigma}{2} \times \frac{T^{\eta, \mathrm v}_n}{\sqrt n} \overset{\mathrm{(d)}}{\underset{n \rightarrow \infty} \longrightarrow} \mathcal T_{\mathrm{Br}}
\end{equation}
where the continuous tree $\mathcal T_{\mathrm{Br}}$ arising in the limit is the Brownian Continuum Random Tree,  sometimes simply called the Brownian tree. Note that the limit only depends on $\eta$ via its variance $\sigma^2$.

This result by Aldous was a breakthrough in the study of large random trees, since it was the first to describe the behavior of the tree as a whole. We will discuss this in more details in Section \ref{SCGW}. Let us first introduce the topological framework in  order to make sense of this convergence.

\subsection{Real Trees and the Gromov--Hausdorff topology}

Since the pioneering works of Evans, Pitman and Winter \cite{EPW06} in 2003 and Duquesne and Le Gall \cite{DLG05} in 2005, the theory of \emph{real trees} (or $\mathbb R$-trees) has been intensively used in probability.  These trees are metric spaces having a ``tree property" (roughly, this means that for each pair of points $x,y$ in the metric space, there is a unique path going from $x$ to $y$ -- see below for a precise definition). This point of view allows behavior such as infinite total length of the tree, vertices with infinite degree, and density of the set of leaves. 


In these lectures, all the real trees we will deal with are compact metric spaces.  For this reason, we restrict ourselves to the theory of compact real trees. We now briefly recall background on real trees and  the Gromov--Hausdorff and Gromov--Hausdorff--Prokhorov distances, and refer to \cite{Eva08,LG06} for more details on this topic. 

\medskip

\textbf{Real trees.} A \emph{real tree} is a metric space $(\T,d)$ such that, for any points $x$ and $y$ in $\T$, 
\begin{enumerate}
\item[$\bullet$] there is an isometry $\varphi_{x,y}:[0,d(x,y)]\to \mathcal T$ such that $
\varphi_{x,y}(0)=x$ and $\varphi_{x,y}(d(x,y))=y$
\item[$\bullet$] for every continuous, injective function $c:[0,1]\to \mathcal T$ with
  $c(0)=x,$ $c(1)=y$, one has $c([0,1])=\varphi_{x,y}([0,d(x,y)])$.
\end{enumerate}

Note that a discrete tree may be seen as a real tree by ``replacing" its edges by line segments. Unless specified, it will be implicit in all these notes that these line segments are all of length 1.  

We denote by $[[x,y]]$ the line segment $\varphi_{x,y}([0,d(x,y)])$ between $x$ and $y$. A \emph{rooted} real tree is an ordered pair $((\mathcal T,d),\rho)$ such that $(\mathcal T,d)$ is a real tree and $\rho \in \mathcal T$. This distinguished point $\rho$ is called the root. 
The height of a point $x\in\T$ is defined by $$\mathsf{ht}(x)=d(\rho,x)$$ and the height of the tree itself is the supremum of the heights of its points, while the diameter is the supremum of the distance between two points:
$$
\mathrm{ht}(\mathcal T)=\sup_{x \in \mathcal T} d(\rho,x) \quad \quad \quad \mathrm{diam}(\mathcal T)=\sup_{x,y \in \mathcal T} d(x,y).
$$ 
The \emph{degree} of a point $x$ is the number of connected components of $\mathcal T \backslash \{x\}$. We call \emph{leaves} of $\T$ all the points which have degree $1$, excluding the root. 
Given two points $x$ and $y$, we define  $x\wedge y$ as the unique point of $\mathcal T$ such that $[[\rho,x]]\cap [[\rho,y]]=[[\rho,x \wedge y]]$. It is called the \emph{branch point} of $x$ and $y$ if its degree is larger or equal to 3. For $a>0$, we define the rescaled tree $a\T$ as $(\T,ad)$ (the metric $d$ thus being implicit and dropped from the notation). 

\bigskip

As mentioned above, we will only consider compact real trees.
We now want to measure how close two such metric spaces are. We start by recalling the definition of Hausdorff distance between compact subsets of a metric space.

\bigskip

\textbf{Hausdorff distance.} If $A$ and $B$ are two nonempty compact subsets of a metric space $(E,d)$,  the Hausdorff distance between $A$ and $B$ is defined by 
	\[d_{E,\mathrm{H}}(A,B)=\inf \big\{\varepsilon>0\:;\: A\subset B^\varepsilon\text{ and }B\subset A^\varepsilon \big\},
\]
where $A^{\varepsilon}$ and $B^{\varepsilon}$ are the closed $\varepsilon$-enlargements of $A$ and $B$, i.e. $A^{\varepsilon}=\{x \in E:d(x,A) \leq \varepsilon\}$ and similarly for $B^{\varepsilon}$.  

The Gromov--Hausdorff extends this concept to compact real trees (or more generally compact metric spaces) that are not necessarily compact subsets of a single metric space, by considering embeddings in common metric spaces.

\bigskip

\textbf{Gromov--Hausdorff distance.} Given two compact rooted trees $(\T,d,\rho)$ and $(\T',d',\rho')$, let
	\[d_{\mathrm{GH}}(\T,\T') = \inf \big\{ \max \big(d_{\mathcal{Z},\mathrm{H}} (\phi(\T),\phi'(\T')), d_\mathcal{Z}(\phi(\rho),\phi'(\rho'))\big)\big\},
\]
where the infimum is taken over all pairs of isometric embeddings $\phi$ and $\phi'$ of $\T$ and $\T'$ in the same metric space $(\mathcal Z,d_{\mathcal Z}),$ for all choices of metric spaces $(\mathcal Z,d_{\mathcal Z})$. 

We will also be concerned with \textit{measured} trees, that are real trees equipped with a probability measure on their Borel sigma--field. To this effect, recall first the definition of the Prokhorov distance between two probability measures $\mu$ and $\mu'$ on a metric space $(E,d)$:
	\[d_{E,\mathrm P}(\mu,\mu')=\inf \big\{\varepsilon >0\:;\: \forall A\in\mathcal{B}(E), \mu(A)\leq\mu'(A^{\varepsilon})+\varepsilon\text{ and } \mu'(A)\leq\mu(A^{\varepsilon})+\varepsilon \big\}.
\]
This distance metrizes the weak convergence on the set of probability measures on $(E,d)$. 

\bigskip

\textbf{Gromov--Hausdorff--Prokhorov distance.} Given two measured compact rooted trees $(\T,d,\rho,\mu)$ and $(\T',d',\rho',\mu')$, we let
  \[d_{\mathrm{GHP}}(\T,\T') = \inf \big\{ \max \big(d_{\mathcal{Z},\mathrm{H}} (\phi(\T),\phi'(\T')), d_\mathcal{Z}(\phi(\rho),\phi'(\rho')),d_{\mathcal{Z}, \mathrm P}(\phi_*\mu,\phi'_*\mu')\big)\big\},
\]
where the infimum is taken on the same space as before and $\phi_*\mu$, $\phi'_*\mu'$ are the push-forwards of $\mu$, $\mu'$ by $\phi$, $\phi'$.

\bigskip

The Gromov--Hausdorff distance $d_{\mathrm{GH}}$ indeed defined a distance on the set  of compact rooted real trees taken up to root--preserving isomorphisms. Similarly, The Gromov--Hausdorff--Prokhorov distance $d_{\mathrm{GHP}}$ is a distance on the set  of compact measured rooted real trees taken up to root--preserving and measure--preserving isomorphisms. Moreover these two metric spaces are Polish, see  \cite{EPW06} and \cite{ADH}. 
We will always identify two (measured) rooted $\mathbb R$-trees when their are isometric and still use the notation $(\mathcal T,d)$ (or $\mathcal T$ when the metric is clear) to design their isometry class.

\bigskip

\textbf{Statistics.} It is easy to check that the function that associates to a compact rooted tree its diameter is continuous (with respect to the GH--topology on the set of compact rooted real trees and the usual topology on $\mathbb R$). Similarly, the function that associates to a compact rooted tree its height is continuous.
The function that associates to a compact rooted measured tree the distribution of the height of a leaf chosen according to the probability on the tree is continuous as well (with respect to the GHP--topology on the set of compact rooted measured real trees and the weak topology on the set of probability measures on $\mathbb R$). Consequently, the existence of scaling limits with respect to the GHP--topology will directly imply scaling limits for the height, the diameter and the height of a typical vertex of the trees.

\subsection{Scaling limits of conditioned Galton--Watson trees}
\label{SCGW}

We can now turn to rigorous statements on the scaling limits of conditioned Galton--Watson trees.
We reformulate the above result  (\ref{cvAldous}) by Aldous in the finite variance case and also present the result by Duquesne \cite{Duq03} when the offspring distribution $\eta$ is heavy tailed, in the domain of attraction of a stable distribution. In the following, $\eta$ always denotes a critical offspring distribution, $T^{\eta,\mathrm v}_n$ is a $\eta$-GW tree conditoned to have $n$ vertices, and $\mu^{\eta, \mathrm v}_n$ is the uniform probability on its vertices.
The following convergences hold for the Gromov--Hausdorff--Prokhorov topology.

\medskip

\begin{thm}
\label{thm:AldousDuquesne}
\begin{enumerate}
\item[\emph{(i)}] \emph{(Aldous \cite{Ald93})} Assume that $\eta$ has a finite variance $\sigma^2$. Then, there exists a random compact real tree, called the Brownian tree and denoted $\mathcal T_{\mathrm{Br}}$, endowed with a probability measure $\mu_{\mathrm{Br}}$ supported by its set of leaves, such that as $n\rightarrow \infty$
$$ \left(\frac{\sigma T^{\eta,\mathrm v}_n}{2\sqrt n},\mu_n^{\eta,\mathrm v}\right) \overset{\mathrm{(d)}}{\underset{\mathrm{GHP}} \longrightarrow} \left(\mathcal T_{\mathrm{Br}},\mu_{\mathrm{Br}}\right).
$$
\item[\emph{(ii)}] \emph{(Duquesne \cite{Duq03})} If $\eta_k \sim \kappa k^{-1-\alpha}$ as $k \to \infty$ for $\alpha \in (1,2)$, then there exists a random compact real tree $\mathcal T_{\alpha}$, called the \emph{stable Lévy tree with index $\alpha$}, endowed with a probability measure $\mu_{\alpha}$ supported by its set of leaves, such that as $n \rightarrow \infty$
\begin{equation*}
\label{cvGW2}
\left(\frac{T^{\eta,\mathrm v}_n}{n^{1 - 1/\alpha}}, \mu^{\eta,\mathrm v}_n\right) \overset{\mathrm{(d)}}{\underset{\mathrm{GHP}} \longrightarrow} \left(\left(\frac{\alpha(\alpha-1)}{\kappa \Gamma(2-\alpha)} \right)^{1/\alpha} \alpha^{1/\alpha -1} \cdot \mathcal{T}_\alpha,\mu_{\alpha}\right).
\end{equation*}
\end{enumerate}
\end{thm}

The result by Duquesne actually extends to cases where the offspring distribution $\eta$ is in the domain of attraction of a stable distribution with index $\alpha \in (1,2]$. See \cite{Duq03} for details.

\bigskip

The Brownian tree was first introduced by Aldous in the early 90s in the series of papers \cite{Ald91,Ald91a,Ald93}. This tree can be constructed in several ways, the most common being the following. Let $(\mathbf e(t),t \in [0,1])$ be a normalized Brownian excursion, which, formally, can be defined from a standard Brownian motion $B$ by letting 
$$
\mathbf e(t)=\frac{\left| B_{g+t(d-g)}\right|}{\sqrt{d-g}}, \quad 0 \leq t \leq 1,
$$
where $g:=\sup\{s \leq 1:B_s=0\}$ and $d=\inf\{s \geq 1:B_s=0\}$ (note that $d-g>0$ a.s. since $B_1 \neq 0$ a.s.). Then consider for $x,y \in [0,1], x \leq y$, the non--negative quantity
$$
d_{\mathbf e}(x,y)=\mathbf e(x)+\mathbf e(y)-2 \inf_{z \in [x,y]} \{\mathbf e(z)\},
$$
and then the equivalent relation $x \sim_{\mathbf e}y \Leftrightarrow d_{\mathbf e}(x,y)=0$.
It turns out that the quotient space $[0,1]/\sim_{\mathbf e}$ endowed with the metric induced by $d_{\mathbf e}$ (which indeed gives a true metric) is a compact real tree. The Brownian excursion $\mathbf e$ is called the \emph{contour function} of this tree.  Equipped with the  measure $\mu_{\mathbf e}$, which is the push--forward of the Lebesgue measure on $[0,1]$, this gives a version $([0,1]/\sim_{\mathbf e}, d_{\mathbf e}, \mu_{\mathbf e})$ of the measured tree $(\mathcal T_{\mathrm Br},\mu_{\mathrm Br})$.
To get a better intuition of what this means, as well as more details and other constructions of the Brownian tree, we refer to the three papers by Aldous \cite{Ald91,Ald91a,Ald93} and to the survey by Le Gall \cite{LG06}. 
  
In the early 2000s, the family of stable Lévy trees $(\mathcal T_{\alpha},\alpha \in(1,2)]$ -- where by convention $\mathcal T_2$ is  $\sqrt 2 \cdot \mathcal T_{\mathrm{Br}}$ -- was introduced by Duquesne and Le Gall~\cite{DLG02, DLG05} in the more general framework of \emph{Lévy trees}, building on earlier work of Le Gall and Le Jan \cite{LGLJ98}.  These trees can be constructed in a way similar as above from continuous functions built from the stable Lévy processes.  This construction is complex and we will not detail it here. Others constructions are possible, see e.g. \cite{DW07,GH15}.
The stable trees are important objects of the theory of random trees. They are intimately related to continuous state branching processes, fragmentation and coalescence processes. They appear as scaling limits of various models of trees and graphs, starting with the Galton--Watson examples above and  some other examples discussed in Section \ref{sec:applic}. In particular, it is noted that it was only proved recently that Galton--Watson trees conditioned by their number of leaves or more general arbitrary degree restrictions also converge in the scaling limit to stable trees, see Section \ref{GWmore} and the references therein. 

In the last few years, the geometric and fractal aspects of stable trees have been studied in great detail: Hausdorff and packing dimensions and measures  \cite{DLG05, DLG06, Duq12, HM04}; spectral dimension \cite{CrH10};  spinal decompositions and invariance under uniform re-rooting \cite{HPW09,DLG09}; fragmentation into subtrees \cite{Mier03, Mier05}; and embeddings of stable trees into each other \cite{CH13}. We simply point out here that the Brownian tree is \emph{binary}, in the sense that all its points have their degree in $\{1,2,3\}$ almost surely, whereas the stable trees  $\mathcal T_{\alpha},\alpha \in(1,2)$ have only points with degree in $\{1,2,\infty\}$ almost surely (every branch point has an infinite number of children).

\bigskip

\textbf{Applications to combinatorial trees.} Using the connections between some families of combinatorial trees and Galton--Watson trees mentioned in Section \ref{sec:discrete}, we obtain the following scaling limits (in all cases, $\mu_n$ denotes the uniform probability on the vertices of the tree $T_n$): 
\begin{enumerate}
\item[$\bullet$] If $T_n$ is uniform amongst the set of rooted ordered trees with $n$ vertices, $$\left(\frac{T_n}{\sqrt n},\mu_n\right) \overset{\mathrm{(d)}}{\underset{\mathrm{GHP}}\longrightarrow} \left(\mathcal T_{B_r},\mu_{\mathrm{Br}}\right).$$
\item[$\bullet$] If $T_n$ is uniform amongst the set of rooted  trees with $n$ labelled vertices, $$\left(\frac{T_n}{\sqrt n},\mu_n\right) \overset{\mathrm{(d)}} {\underset{\mathrm{GHP}}\longrightarrow} \left(2\mathcal T_{B_r},\mu_{\mathrm{Br}}\right).$$
\item[$\bullet$] If $T_n$ is uniform amongst the set of rooted binary ordered trees with $n$ vertices, $$\left(\frac{T_n}{\sqrt n},\mu_n\right) \overset{\mathrm{(d)}} {\underset{\mathrm{GHP}}\longrightarrow} \left(2\mathcal T_{B_r},\mu_{\mathrm{Br}}\right).$$
\end{enumerate}
As a consequence, this provides the behavior of several statistics of the trees, that first interested combinatorists.

\bigskip 

We will not present the proofs of Aldous \cite{Ald93} and Duquesne \cite{Duq03} of their results, but will rather focus on the fact that they may be recovered  by using the MB--property. This is the goal of the next two sections, where we will present in a general setting some results on the scaling limits for MB--sequences of trees. As already mentioned, the main idea of the proofs of Aldous \cite{Ald93} and Duquesne \cite{Duq03} is rather based on the study of the so--called \emph{contour functions} of the trees. We refer to Aldous and Duquesne papers, as well as Le Gall's survey \cite{LG06} for details. See also Duquesne and Le Gall \cite{DLG02} and Kortchemski \cite{Kort12, KortSimple} for further related results.

\section{Scaling limits of Markov--Branching trees}
\label{SCMB}

Our goal is to set up an asymptotic criterion on the splitting probabilities $(q_n)$ of a MB--sequence of trees so that this sequence, suitably normalized, converges to a non--trivial  continuous limit. We follow here the approach of the paper \cite{HM12} that found its roots in the previous work \cite{HMPW08} were similar results where proved under stronger assumptions. A remark on these previous results is made at the end of this section.

The splitting probability $q_n$ corresponds to a ``discrete" fragmentation of the integer $n$ into smaller integers. To set up the desired criterion, we first need to introduce a continuous counterpart for these partitions of integers, namely
$$\textcolor{black}{\mathcal S^{\downarrow}=\left\{\mathbf s=(s_1,s_2,\ldots): s_1 \geq s_2 \geq ... \geq 0 \ \text{and} \ \sum_{i\geq 1} s_i = 1\right\}}$$
which is endowed with the distance $d_{\mathcal S^{\downarrow}}(\mathbf s,\mathbf s')=\sup_{i\geq 1}|s_i-s_i'|$.
Our main hypothesis on $(q_n)$ then reads:

\smallskip

\begin{mybox3}
\textbf{Hypothesis ($\mathsf H$)}: there exist  \textcolor{black}{$\gamma>0$} and \textcolor{black}{$\nu$} a non--trivial \textcolor{black}{$\sigma-$finite measure}  on $\mathcal S^{\downarrow}$ satisfying \newline $\int_{\mathcal S^{\downarrow}}(1-s_1)\nu(\mathrm d \mathbf s)<\infty$ and $\nu(1,0,\ldots)=0$, such that
$$
n^{\textcolor{black}{\gamma}} \sum_{\lambda \in \mathcal P_n} q_n\left(\lambda\right) \Big(1-\frac{\lambda_1}{n} \Big)f\Big( \frac{\lambda_1}{n},\ldots,\frac{\lambda_p}{n},0,\ldots\Big) \underset{n \rightarrow \infty}\longrightarrow \int_{\mathcal S^{\downarrow}}(1-s_1) f(\mathbf s)\textcolor{black}{\nu}(\mathrm d \mathbf s).
$$
for all continuous $f: \mathcal S^{\downarrow} \rightarrow \mathbb R$.
\end{mybox3}

\vspace{-0.1cm}

We will see in Section \ref{sec:applic} that most of the examples of splitting probabilities introduced in Section \ref{sec:MB} satisfy this hypothesis. As a first, easy, example, consider the  ``basic example" introduced there (Example 2): $q_n((n))=1-n^{-\alpha}$ \ and \ $q_n(\lceil n/2\rceil, \lfloor n/2\rfloor)=n^{-\alpha}$, $\alpha>0$. Then, clearly, ($\mathsf H$) is satisfied with
$$
\gamma=\alpha \quad \text{and} \quad \nu(\mathrm d \mathbf s)=\delta_{\left(\frac{1}{2},\frac{1}{2},0,\ldots \right)}.
$$
The interpretation of the hypothesis ($\mathsf H$)  is that macroscopic branchings are rare,
in the sense that the macroscopic splitting events $n \mapsto n\mathbf s$, $\mathbf s \in \mathcal S^{\downarrow}$ with $s_1<1-\varepsilon$ occur  with a  probability asymptotically proportional to $n^{-\gamma} \mathbbm 1_{\{s_1 <1-\varepsilon \}}\nu(\mathrm d \mathbf s)$, for a.e. fixed $\varepsilon \in (0,1)$.
\bigskip

The main result on the scaling limits of MB--trees indexed by the leaves is the following.

\begin{thm}[\cite{HM12}]
\label{mainthm:leaves}
Let $(T_n,n \geq 1)$ be a \emph{MB}--sequence indexed by the leaves and assume that its splitting probabilities satisfy $\mathsf{(H)}$. Then there exists a compact, measured real tree $(\mathcal T_{\gamma,\nu}, \mu_{\gamma,\nu})$ such that
$$
\textcolor{black}{ \left(\frac{T_n}{n^{\gamma}}, \mu_n \right) \underset{\mathrm{GHP}}{\overset {\mathrm{(d)}} \longrightarrow} \left(\mathcal T_{\gamma,\nu}, \mu_{\gamma,\nu}\right)},
$$
where $\mu_n$ is the uniform probability on the leaves of $T_n$. 
\end{thm}

\smallskip
The goal of this section is to detail the main steps of the proof of this result and to discuss some properties  of the limiting measured tree, which belongs to the so--called family of \emph{self--similar fragmentation trees} (the distribution of such a tree is entirely characterized by the parameters $\gamma$ and $\nu$). In that aim we will first study how  the height of a leaf chosen uniformly at random in $T_n$ grows (Section \ref{sec:Markov} and Section \ref{sec:limitMarkov}). Then we will review some results on self--similar fragmentation trees (Section \ref{sec:fragtrees}). Last we will explain how one can use the scaling limit  of the height of a leaf chosen at random to obtain, by induction, the scaling limit  of the subtree spanned by $k$ leaves chosen independently, for all $k$, and then finish the proof of Theorem \ref{mainthm:leaves} with a tightness criterion (Section \ref{sec:proof}).

\bigskip

There is a similar result for MB--sequences indexed by the vertices. 

\begin{thm}[\cite{HM12}]
\label{mainthm:vertices}
Let $(T_n,n \geq 1)$ be a \emph{MB}--sequence indexed by the vertices and assume that its splitting probabilities satisfy $\mathsf{(H)}$ for some $0<\gamma<1$.
Then there exists a compact, measured real tree $(\mathcal T_{\gamma,\nu}, \mu_{\gamma,\nu})$ such that
$$
\textcolor{black}{ \left(\frac{T_n}{n^{\gamma}}, \mu_n \right) \underset{\mathrm{GHP}}{\overset {\mathrm{(d)}} \longrightarrow} \left(\mathcal T_{\gamma,\nu}, \mu_{\gamma,\nu}\right)},
$$
where $\mu_n$ is the uniform probability on the vertices of $T_n$. 
\end{thm}

Theorem \ref{mainthm:vertices} is actually a direct corollary of Theorem \ref{mainthm:leaves}, for the following reason. Consider a MB--sequence indexed by the vertices with splitting probabilities $(p_n)$ and for all $n$, branch on each internal vertex of the tree $T_n$ an edge with a leaf. This gives a tree $\overline T_n$ with $n$ vertices. It is then obvious that $(\overline T_n,n\geq 1)$ is a MB--sequence indexed by the \emph{leaves}, with splitting probabilities $(q_n)$ defined by
$$
q_n(\lambda_1,\ldots,\lambda_{p},1)=p_{n-1}(\lambda_1,\ldots,\lambda_{p}), \quad \text{for all } (\lambda_1,\ldots,\lambda_{p}) \in \mathcal P_{n-1}
$$
(and $q_n(\lambda)=0$ for all other $\lambda \in \mathcal P_n$). It is moreover easy to see that $(q_n)$ satisfies ($\mathsf H$) with parameters $(\gamma,\nu)$, $0<\gamma<1$, if and only if $(p_n)$ does. Hence Theorem \ref{mainthm:leaves} implies Theorem \ref{mainthm:vertices}.

\bigskip

We will present in Section \ref{sec:applic} several applications of these two theorems. Let us just consider here the ``basic example" of Section \ref{sec:MB} (Example 2). We have already noticed  that its splitting probabilities satisfy Hypothesis $(\mathrm H)$, with parameters $\alpha$ and $\delta_{(1/2,1/2,\ldots)}$. Hence in this case, the corresponding sequence of MB--trees $T_n$ divided by $n^{\alpha}$ and endowed with the uniform probability measure on its leaves converges for the GHP--topology towards a $(\alpha, \delta_{(1/2,1/2,\ldots)})$--self--similar fragmentation tree. 

\bigskip

\textbf{Remark.} These two statements are also valid when replacing in ($\mathsf H$) and in the theorems the power sequence $n^{\gamma}$ by any regularly varying sequence with index $\gamma>0$. We recall that a sequence $(a_n)$ is said to vary regularly with index $\gamma>0$ if for all $c>0$,
$$
\frac{a_{\lfloor cn \rfloor}}{a_n} \underset{n \rightarrow \infty} \longrightarrow c^{\gamma}.
$$
We refer to \cite{BGT} for backgrounds on that topic. For simplicity, in the following we will only works with power sequences, but the reader should have in mind that everything holds similarly for regularly varying sequences. 

\bigskip

\textbf{Convergence in probability.} In \cite{HMPW08}, scaling limits are established for some MB--sequences that moreover satisfy a property of \emph{sampling consistency}, namely that for all $n$, $T_n$ is distributed as the tree with $n$ leaves obtained by removing a leaf picked uniformly at random in $T_{n+1}$, as well as the adjacent edge. This consistency property is demanding and the approach developed in \cite{HM12} allows to do without it. However a strengthened  version of the consistency property has also a significant advantage, leading to convergence in probability. Indeed, if the MB--sequence is \emph{strongly sampling consistent}, which means that versions of the trees can be built on a same probability space so that if $T_n^{\circ}$ denotes the tree with $n$ leaves obtained by removing a leaf picked uniformly at in $T_{n+1}$, as well as the adjacent edge, then $(T_n,T_{n+1})$ is distributed as $(T^{\circ}_n,T_{n+1})$, then one can establish under suitable conditions the \emph{convergence in probability} of the rescaled trees. See \cite{HMPW08} for details.

\subsection{A Markov chain in the Markov--Branching sequence of trees}
\label{sec:Markov}

Consider $(T_n,n\geq 1)$  a MB--sequence of trees indexed by the leaves, with splitting distribution $(q_n,n\geq 1)$. Before studying the scaling limit of the trees in their whole, we start by studying the scaling limit of a \emph{typical leaf}. I.e., in each $T_n$, we mark one of the $n$ leaves uniformly at random and we want to determine how the height of the marked leaf behaves as $n \rightarrow \infty$. In that aim, let  $\star_n$ denote this marked leaf and  $\star_n(k)$ denote its ancestor at generation $k$, $0\leq k \leq n$ (so that $\star_n(0)$ is the root of $T_n$ and $\star_n({\mathrm{ht}(\star_n)})=\star_n$). Let also $T^{\star}_n(k)$ be the subtree composed by the descendants of $\star_n(k)$ in $T_n$, formally,
$$
T^{\star}_n(k):=\left\{v \in T_n : \star_n(k)\in [[\rho,v]]\right\}, \quad k \leq \mathrm{ht}(\star)
$$
and $T^{\star}_n(k):=\emptyset$ if $k>\mathrm{ht}(\star)$. We then set
\begin{equation}
\label{aMC}
X_n(k):=\#\left\{\text{leaves of }T^{\star}_n(k)\right\}, \quad \forall k \in \mathbb Z_+
\end{equation}
with the convention that $X_n(k)=0$ for $k >\mathrm{ht}(\star)$.

\begin{figure}
\begin{center}
\includegraphics[height=2.7cm]{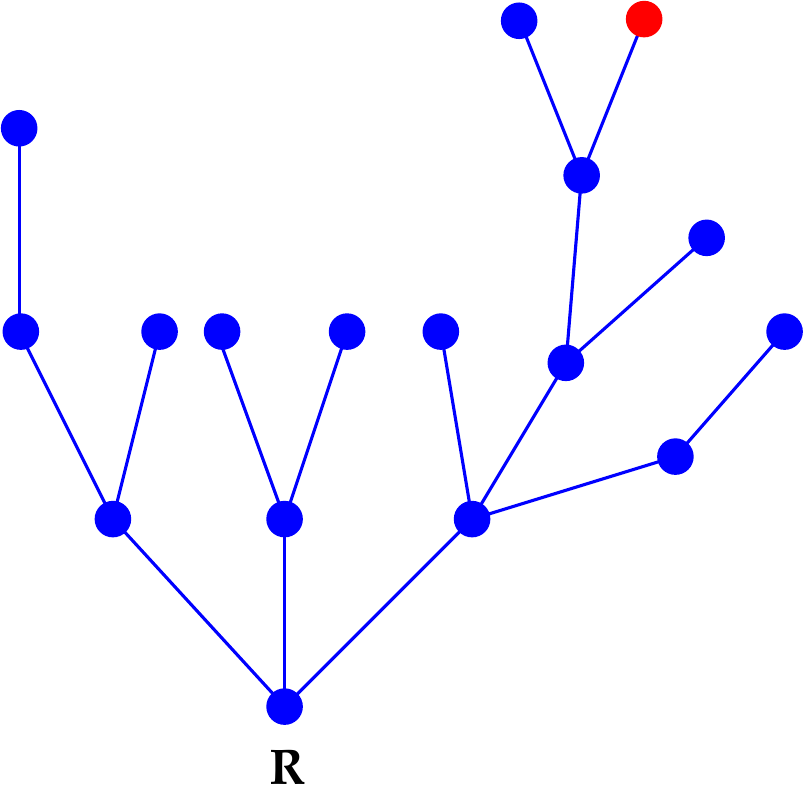}
\hspace{0.15cm}
\includegraphics[height=2.7cm]{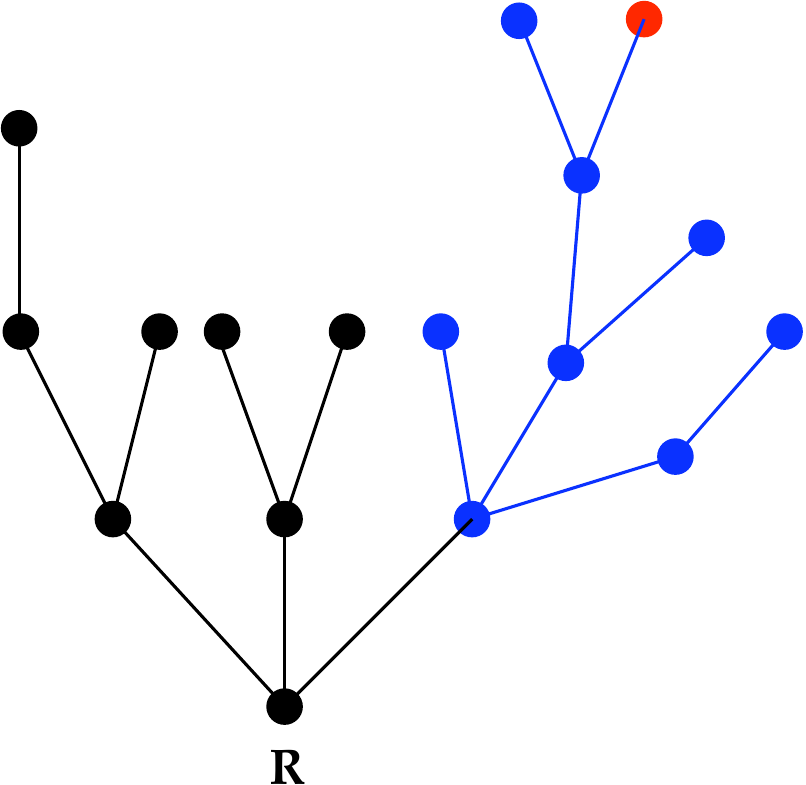}
\hspace{0.15cm}
\includegraphics[height=2.7cm]{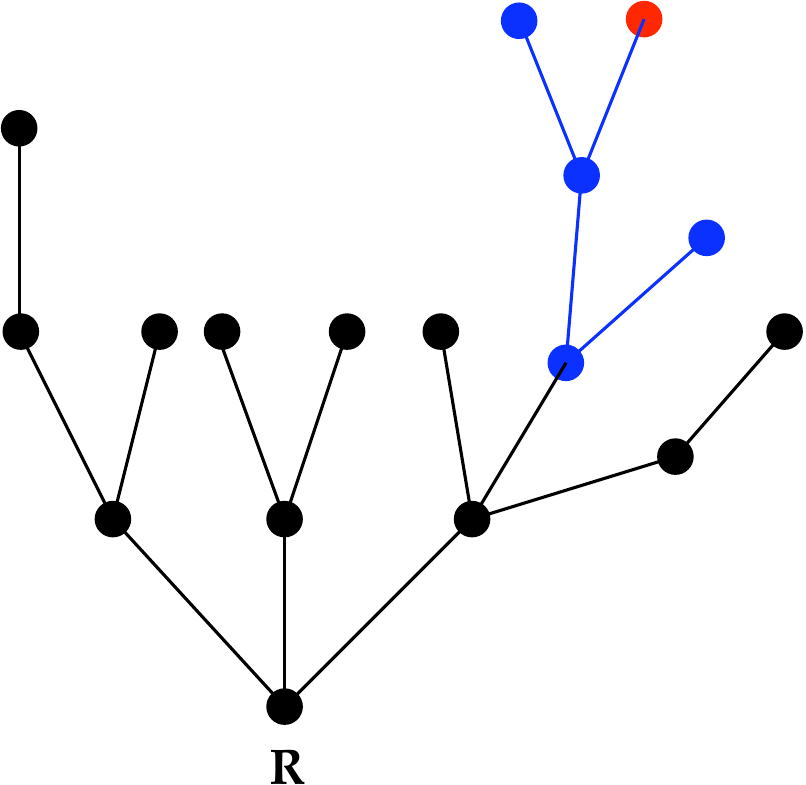}
\hspace{0.15cm}
\includegraphics[height=2.7cm]{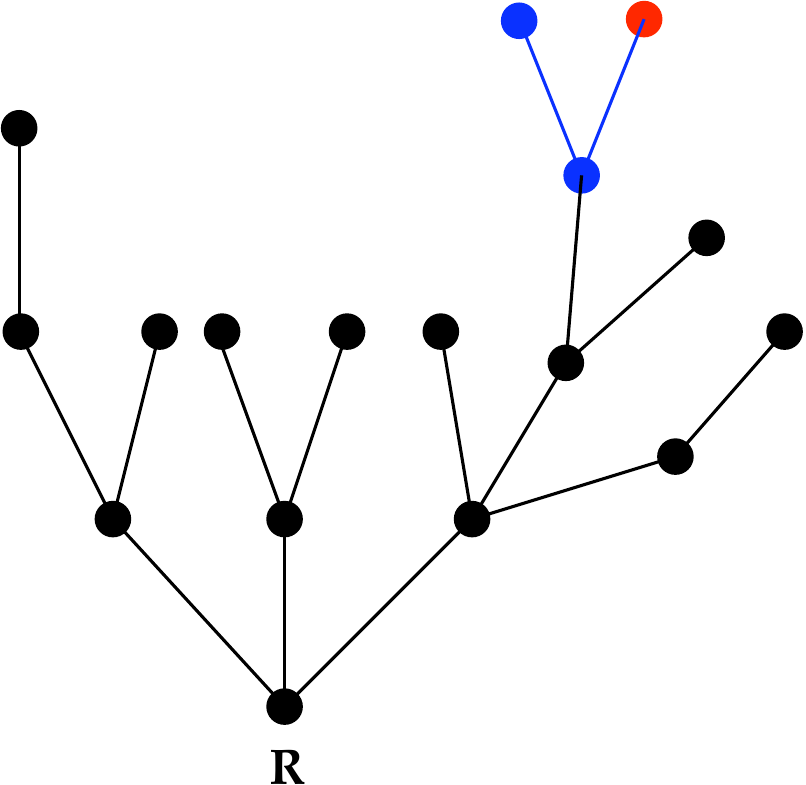}
\hspace{0.15cm}
\includegraphics[height=2.7cm]{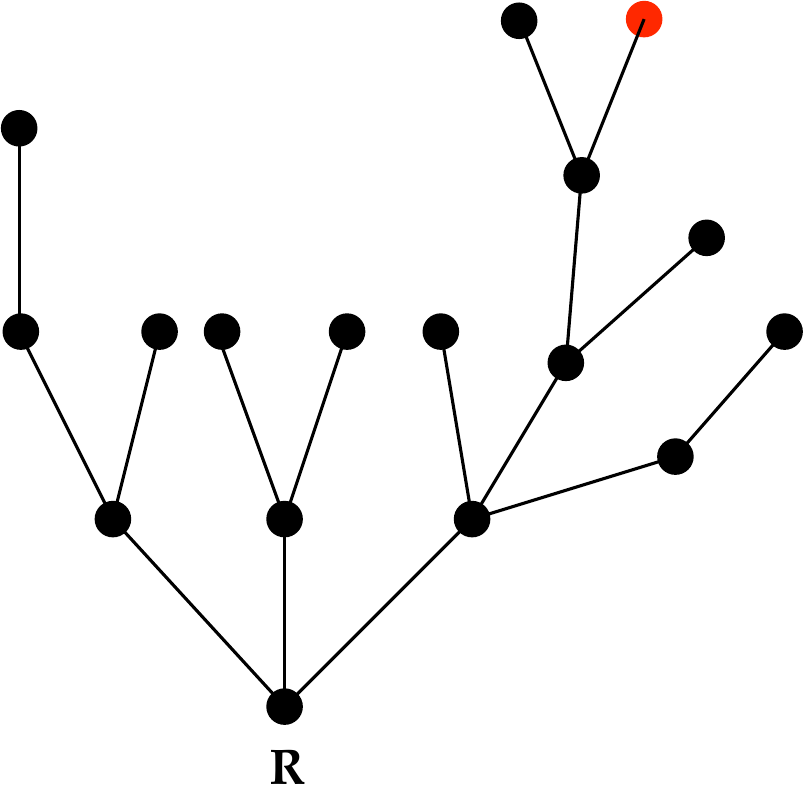}
\hspace{0.15cm}
\caption{A Markov chain in the Markov--Branching trees: here, $n=9$ and $X_9(0)=9$, $X_9(1)=5$, $X_9(2)=3$, $X_9(3)=2$, $X_9(4)=1$ and $X_9(i)=0, \forall i \geq 5$.}
\end{center}
\end{figure}

\begin{prop}
The process $(X_n(k),k \geq 0)$ is a $\mathbb Z_+$--valued non--increasing Markov chain starting from $X_n(0)=n$, with transition probabilities 
\begin{equation}
\label{transitionproba}
p(i,j)=\sum_{\lambda \in \mathcal P_i}q_i(\lambda)m_ j(\lambda)\frac{j}{i} \quad \text{ for all } 1\leq j \leq i, \text{ with }i \geq 1
\end{equation}
and $p(1,0)=q_1(\emptyset)=1-p(1,1)$.
\end{prop}

\begin{proof}
The Markov property is a direct consequence of the Markov branching property. Indeed, given $X_n(1)=i_1,\ldots,X_n(k-1)=i_{k-1}$, the tree $T^{\star}_n(k-1)$ is distributed as $T_{i_{k-1}}$ if $i_{k-1} \geq 1$ and is the emptyset otherwise. In particular, when $i_{k-1}=0$, the conditional distribution of $X_{n}(k)$ is the Dirac mass at 0. When $i_{k-1}\geq 1$, we use the fact that $\star_n$ is in $T^{\star}_n(k-1)$, hence, still conditioning on the same event,  we have that $\star_n$ is uniformly distributed amongst the $i_{k-1}$ leaves of $T_{i_{k-1}}$. Otherwise said, given $X_n(1)=i_1,\ldots,X_n(k-1)=i_{k-1}$ with $i_{k-1} \geq 1$, $(T^{\star}_n(k-1), \star_n)$ is distributed as $(T_{i_{k-1}},\star_{i_{k-1}})$ and consequently $X_n(k)$ is distributed as $X_{i_{k-1}}(1)$. Hence the Markov property of the chain $(X_n(k),k \geq 0)$.
It remains to compute the transition probabilities:
$$
p(n,k)=\mathbb P(X_n(1)=k)=\sum_{\lambda \in \mathcal P_n} q_n(\lambda) \mathbb P\left(X_n(1)=k | \Lambda_n=\lambda\right)
$$
where $\Lambda_n$ denotes the partition of $n$ corresponding to the distribution of the leaves in the subtrees of $T_n$ above the root. Since $\star_n$ is chosen uniformly amongst the set of leaves, we clearly have that
$$
\mathbb P\left(X_n(1)=k | \Lambda_n=\lambda\right)=\frac{k}{n} \times \#\{j:\lambda_j=k\}, \quad \forall k \geq 1.
$$
\end{proof}

Hence studying the scaling limit of  the height of the marked leaf in the tree $T_n$ reduces to studying the scaling limit of the absorption time $A_n$ of the Markov chain 
$(X_n(k),k \geq 0)$ at 0:
$$
A_n:=\inf \big\{k\geq 0:X_n(k)=0 \big\}
$$
(to be precise, this absorption time is equal to the height of the marked leaf $+1$). The study of the scaling limit of $\left((X_n(k),k \geq 1),A_n\right)$ as $n \rightarrow \infty$ is the goal of the next section. Before getting in there, let us notice that the Hypothesis $(\mathrm H)$ on the splitting probabilities $(q_n,n\geq 1)$ of $(T_n,n\geq 1)$, together with (\ref{transitionproba}), implies the following behavior of the transition probabilities $(p(n,k),k\leq n)$:
\begin{equation}
\label{cvprobaheight}
n^{\gamma} \sum_{k=0}^n p(n,k) \left(1-\frac{k}{n}\right)g\left(\frac{k}{n}\right) \underset{n \rightarrow \infty}\longrightarrow \int_{[0,1]} g(x) \mu (\mathrm dx)
\end{equation}
for all continuous functions $g:[0,1] \rightarrow \mathbb R$, where the measure $\mu$ in the limit is a finite, non--zero measure on $[0,1]$ defined by 
\begin{equation}
\label{defmuheight}
\int_{[0,1]} g(x) \mu (\mathrm dx)= \int_{\mathcal S^{\downarrow}} \sum_{i \geq 1}s_i (1-s_i)g(s_i) \nu(\mathrm d \mathbf s).
\end{equation}
To see this, apply $(\mathrm H)$ to the continuous function defined by
$$
f(\mathbf s)=\frac{\sum_{i\geq 1} s_i(1-s_i) g(s_i)}{1-s_1} \quad \text{for }\mathbf s \neq (1,0,\ldots)
$$ 
and $f(1,0,\ldots)=g(1)+g(0)$.

\subsection{Scaling limits of non--increasing Markov chains}
\label{sec:limitMarkov}

As discussed in the previous section, studying the height of a typical leaf in MB--trees amounts to studying the absorption time at 0 of a $\mathbb Z_+$--valued  non-increasing Markov chain. In this section, we study in a general framework the scaling limits of $\mathbb Z_+$--valued  non--increasing Markov chains, under appropriate assumptions on the transition probabilities. At the end of the section we will see how this applies to the height of a typical leaf in a MB--sequence.
In the following,
$$
\left(X_n(k),k \geq 0\right)
$$
denotes a non--increasing $\mathbb Z_+$--valued  Markov chain starting from $n$ ($X_n(0)=n$), with transition probabilities $(p(i,j), 0 \leq j \leq i)$ such that

\medskip

\begin{mybox3}
\textbf{Hypothesis ($\mathsf H'$)}: $\exists$ \textcolor{black}{$\gamma>0$} and \textcolor{black}{$\mu$} a non--trivial \textcolor{black}{finite measure}  on $[0,1]$ such that
$$
n^{\gamma} \sum_{k=0}^n p(n,k) \left(1-\frac{k}{n}\right)f\left(\frac{k}{n}\right) \underset{n \rightarrow \infty}{\longrightarrow} \int_{[0,1]} f(x) \mu (\mathrm dx)
$$
for all continuous functions $f:[0,1] \rightarrow \mathbb R.$
\end{mybox3}

\vspace{-0.2cm}

This hypothesis implies that starting from $n$, \emph{macroscopic jumps} (i.e. with size proportional to $n$) are rare, since for a.e. $0<\varepsilon \leq 1$, the probability to do a jump larger than $\varepsilon n$ is of order $c_{\varepsilon}n^{-\gamma}$ where 
$c_{\varepsilon}= \int_{[0,1-\varepsilon]} (1-x)^{-1}\mu (\mathrm dx)$
(note that this may tend to $\infty$ when $\varepsilon$ tends to 0). 

Now, let
$$
A_n:=\inf\big\{k \geq 0: X_n(i)=X_n(k), \quad \forall i \geq k\big\}
$$
be the first time at which the chain enters an absorption state (note that $A_n<\infty$ a.s. since the chain is non-increasing and $\mathbb Z_+$--valued). In the next theorem, $\mathbb D([0,\infty),[0,\infty))$ denotes the set of non--negative càdlàg processes, endowed with the Skorokhod topology. 

\begin{thm}[\cite{HM11}]
\label{cvMC}
Assume $(\mathsf{H'})$.

\emph{(i)} Then, in $\mathbb D([0,\infty),[0,\infty))$,
$$
\left( \frac{X_n\left(\lfloor n^{\gamma}t \rfloor \right)}{n}, t \geq 0 \right) \overset{\mathrm{(d)}}{\underset{n \rightarrow \infty} \longrightarrow} \left(\exp(-\xi_{\tau(t)}),t \geq 0 \right),
$$
where 
$\xi$ is a subordinator, i.e. a non--decreasing Lévy process, and $\tau$ the time--change (acceleration of time) $$\tau(t):=\inf \left \{u \geq 0 : \int_0^u \exp(-\gamma \xi_r) \mathrm dr \geq t\right\}, t \geq 0.$$
The distribution of $\xi$ is characterized by its Laplace transform $\mathbb E[\exp(-\lambda \xi_t)]=\exp(-t \phi(\lambda))$,  with
$$
{\phi(\lambda)=\mu(\{0\})}+\mu(\{1\}) \lambda + \int_{(0,1)}(1-x^{\lambda})\frac{\mu(\mathrm dx)}{1-x}, \ \ \lambda \geq 0.
$$

\emph{(ii)} Moreover, jointly with the above convergence,  
$$
\frac{A_n}{n^{\gamma}} \overset{\mathrm{(d)}}{\underset{n \rightarrow \infty} \longrightarrow} \int_0^{\infty}\exp(-\gamma \xi_r) \mathrm dr=\inf \left\{t \geq 0: \exp(-\xi_{\tau(t)})=0 \right\}.
$$
\end{thm}

\smallskip

\textbf{Comments.} For background on Lévy processes, we refer to \cite{BertoinLevy}. Let us simply recall here that the law of a subordinator is characterized by three parameters: a measure on $(0,\infty)$ that codes its jumps (which here is the push--forward of $\mu(\mathrm dx)(1-x)\mathbbm 1_{\{x \in (0,1)\}}$ by the application $x \mapsto -\ln(x)$), a linear drift (here $\mu(\{1\})$) and a killing rate at which the process jumps to $+\infty$ (here $\mu(\{0\})$).

\bigskip

\textbf{Main ideas of the proof of Theorem \ref{cvMC}}. (i) Let $Y_n(t): = n ^{-1}X_n( \lfloor n^{\gamma}t \rfloor)$, for $t\geq 0,n\in \mathbb N$. First, using Aldous' tightness criterion \cite[Theorem 16.10]{Bill99} and $(\mathsf H')$, one can check that the sequence $(Y_n,n \geq 1)$ is tight. 
It is then sufficient to prove that every possible limit in distribution of subsequences of $(Y_n)$ are distributed as $\exp(-\xi_{\tau})$. 
Let \textcolor{black}{$Y'$} be such a limit and  $(n_k, k\geq 1)$ a sequence such that $Y_{n_k}$ converges to $Y'$ in distribution.  In the limit, we actually prefer to deal with $\xi$ than with $\xi_{\tau}$, and for this reason we start by changing time in $Y_n$ by setting 
$$\tau_{Y_n}(t):=\inf \left\{u \geq 0:\int_0^u Y_n^{-\gamma}(r) \mathrm dr>t \right\}
\quad \text{and}  \quad Z_n(t):=Y_n\left(\tau_{Y_n}(t)\right), \quad t \geq 0.
$$
One can then easily check that $(Z_{n_k})$ converges in distribution to $Z'$ where $Z' =Y'\circ \tau_{Y'}$, with $\tau_{Y'}(t):=\inf \big\{u \geq 0:\int_0^u (Y'(r))^{-\gamma} \mathrm dr>t \big\}$. It is also easy to reverse the time--change and get that
$$Y'(t)=Z'\big(\tau^{-1}_{Y'}(t)\big) \ =Z'\left(\inf \left\{u \geq 0:\int_0^u Z'^{\gamma}(r) \mathrm dr>t \right\}\right), \quad t \geq 0.$$
With this last equality, we see that it just remains to prove that $Z'$ is distributed as $\exp(-\xi)$. This can be done in three steps:
\vspace{-0.1cm}
\begin{enumerate}
\item[] (a) Observe the following (easy!) fact: if $P$ is the transition function of a Markov chain $M$ with countable state space $\subset \mathbb R$, then for  any positive function $f$ such that $f^{-1}(\{0\})$ is absorbing,
$$
f(M(k))\prod_{i=0}^{k-1}\frac{f(M(i))}{Pf(M(i))}, \quad  k\geq 0
$$
is a martingale. As a consequence: for all $\lambda \geq 0$ and $n \geq 1$, if we let $G_n(\lambda):=\mathbb E\left[(X_n(1)/n)^{\lambda}\right]$, then,
$$
\textcolor{black}{\textstyle M^{(\lambda)}_n(t)}:=Z^{\lambda}_n(t) \left( \prod_{i=0} ^{\lfloor  n^{\gamma}\tau_{Y_n}(t)   \rfloor-1} G_{X_n(i)}(\lambda)\right)^{-1}, \quad t \geq 0
$$
is a martingale.

\item[] (b) Under $(\mathsf H')$, $1-G_n(\lambda) \underset{n \rightarrow \infty}{\sim} n^{-\gamma} \phi(\lambda)$. Together with the convergence in distribution of $(Z_{n_k})$ to $Z'$ and the definition of $M^{(\lambda)}_n$, this leads to the convergence  (this is the most technical part)
$$M^{(\lambda)}_{n_k} \overset{\mathrm{(d)}}{\underset{k \rightarrow \infty} \longrightarrow} (Z')^{\lambda} \exp(\phi(\lambda) \cdot),$$ and the martingale property passes to the limit.

\item[] (c) Hence $(Z')^{\lambda}\exp(\phi(\lambda) \cdot)$ is a martingale for all $\lambda \geq 0$. Using Laplace transforms, it is then easy to see that this implies in turn that $-\ln Z'$ is a non--decreasing process with independent and stationary increments (hence a subordinator), with Laplace exponent $\phi$.
\end{enumerate}
Hence $Z'\overset{\mathrm{(d)}}=\exp(-\xi)$. 

(ii) We do not detail this part and refer to \cite[Section 4.3]{HM11}. Let us  simply point out that it is \emph{not} a direct consequence of the convergence of $(Y_n)$ to $\exp(-\xi_{\tau})$ since
convergence of functions in $\mathbb D([0,\infty),[0,\infty))$ does not lead, in general, to the convergence of their absorption
times (when they exist). $\hfill \square$

\bigskip

This result leads to the following corollary. 

\begin{cor}
\label{corocvMC}
Let $(T_n,n\geq 1)$ be a \emph{MB}-sequence indexed by the leaves, with splitting probabilities satisfying $\mathsf{(H)}$ with parameters $(\gamma,\nu)$. For each $n$, let $\star_{n}$ be a leaf chosen uniformly amongst the $n$ leaves of $T_n$. Then,
$$
\frac{\mathsf{ht}(\star_{n})}{n^{\gamma}} \overset{\mathrm{(d)}}{\underset{n \rightarrow \infty} \longrightarrow} \int_0^{\infty} \exp(-\gamma \xi_r) \mathrm dr
$$
where $\xi$ is a subordinator with Laplace exponent $\phi(\lambda)=\int_{\mathcal S^{\downarrow}} \sum_{i \geq 1}\left(1-s_i^{\lambda}\right)s_i \nu(\mathrm d \mathbf s), \lambda \geq 0$.
\end{cor}

\begin{proof} 
As seen at the end of the previous section, under ($\mathsf H$) the transition probabilities of the Markov chain (\ref{aMC}) satisfy assumption $(\mathsf{H'})$ with parameters $\gamma$ and $\mu$, with $\mu$ defined by (\ref{defmuheight}). 
The conclusion follows with Theorem \ref{cvMC} (ii).
\end{proof}

\medskip

\textbf{Further reading.} Apart from applications to Markov--Branching trees, Theorem \ref{cvMC} can be used to describe the scaling limits  of various stochastic processes, e.g.
random walks with a barrier or the number of collisions in $\Lambda$--coalescent processes, see \cite{HM11}. Recently, Bertoin and Kortchemski \cite{BK15} set up results similar to Theorem \ref{cvMC} for \emph{non--monotone} Markov chains and develop several applications, to random walks conditioned to stay positive, to the number of particles in some coagulation--fragmentations processes, to random planar maps (see \cite{BCK15} for this last point). Also in \cite{HSMAP} (in preparation) similar convergences for bivariate Markov chains towards time--changed Markov additive processes are studied. This will have applications to dynamical models of tree growth in a broader context than the one presented in Section \ref{sec:cvdynam}, and more generally to multi--type MB--trees.

\subsection{Self--similar fragmentation trees}
\label{sec:fragtrees}

Self--similar fragmentation trees are random compact measured real trees that describe the genealogical structure of self--similar fragmentation processes with a negative index. It turns out that this set of trees is closely related to the set of trees arising as scaling limits of MB--trees. We start by introducing the self--similar fragmentation processes, following Bertoin \cite{BertoinBook}, and then turn to the description of their genealogical trees, which were first introduced in \cite{HM04} and then in \cite{Steph13} in a broader context.

\subsubsection{Self--similar fragmentation processes}

Fragmentation processes are continuous--time processes that describe the evolution of an object that splits repeatedly and randomly as time passes. In the models we are interested in, the fragments are characterized by their mass alone, other characteristics, such as their shape, do not come into account. Many researchers have been working on such models. From a historical perspective, it seems that Kolmogorov \cite{Kolmo41} was the first in 1941. Since the early 2000s, there has been a full treatment of fragmentation processes satisfying a self--similarity property. We refer to Bertoin's book \cite{BertoinBook} for an overview of work in this area and a deepening of the results presented here. 

We will work on the space of masses
$$\textcolor{black}{\mathcal S_{-}^{\downarrow}=\left\{\mathbf s=(s_1,s_2,\ldots): s_1 \geq s_2 \geq ... \geq 0 \ \text{and} \ \sum_{i\geq 1} s_i \leq 1\right\}},$$
which contains the set $\mathcal S^{\downarrow}$, and which is equipped with the same metric $d_{\mathcal S^{\downarrow}}$.

\begin{defn}
Let $\alpha \in \mathbb R$.
An $\alpha-$self--similar fragmentation process is a $\mathcal S^{\downarrow}_{-}$--valued Markov process  $(F(t),t\geq 0)$ which is continuous in probability and such that, for all $t_0\geq 0$, given that $F(t_0)=(s_1,s_2,\ldots)$, the process $(F(t_0+t),t \geq 0)$  is distributed as the process $G$ obtained by considering a sequence $(F^{(i)},i\geq 1)$ of i.i.d. copies of $F$ and then defining $G(t)$ to be the decreasing rearrangement of the sequences $s_iF^{(i)}(s_i^{\alpha}t),i\geq 1$, for all $t\geq 0$.
\end{defn}

In the following, we will always consider processes starting from a unique mass equal to 1, i.e. $F(0)=(1,0,\ldots)$. 
At time $t$, the sequence $F(t)$ should be understood as  the decreasing sequence of the masses of fragments present at that time. 

It turns out that such processes indeed exist and that their distributions are characterized by three parameters: the \emph{index of self--similarity} $\alpha \in \mathbb R$, an \emph{erosion coefficient} $c \geq 0$ that codes a continuous melt of the fragment (when $c=0$ there is no erosion) and a \emph{dislocation measure} $\nu$, which is a measure $\nu$ on  $\mathcal S^{\downarrow}_{-}$ such that $\int_{\mathcal S^{\downarrow}_{-}}(1-s_1)\nu(\mathrm d \mathbf s)<\infty$. The role of the parameters $\alpha$ and $\nu$ can be specified as follows when $c=0$ and $\nu$ is finite: then, each fragment with mass $m$ waits a random time with exponential distribution with parameter $\nu(\mathcal S^{\downarrow}_{-})$ and then splits in fragments with masses $m\mathbf S$, where $\mathbf S$ is distributed according to $\nu/\nu(\mathcal S^{\downarrow}_{-})$, independently of the splitting time. When $\nu$ is infinite, the fragments split immediately, see \cite[Chapter3]{BertoinBook} for further details. 

The index $\alpha$ as an enormous influence on the behavior of the process: when $\alpha=0$, all fragments split at the same rate, whereas when $\alpha>0$ fragments with small masses split slower and  when $\alpha<0$ fragments with small masses split faster. In this last case the fragments split so quickly that the whole initial object is reduced to ``dust" in a finite time, almost surely, i.e. $\inf\{t\geq 0:F(t)=(0,\ldots)\}<\infty$ a.s.

\textbf{The tagged fragment process.} We turn to a connection with the results seen in the previous section. Pick a point uniformly at random in the initial object (this object can be seen as an interval of length 1, for example), independently of the evolution of the process and let $F_{*}(t)$ be the mass of the fragment containing this marked point at time $t$. The process $F_{*}(t)$ is non--increasing and more precisely,

\begin{thm}[Bertoin \cite{BertoinBook}, Theorem 3.2 $\&$ Corollary 3.1]
\label{thm:tagged}

The process $F_{*}$ can be written as
$$
F_{*}(t)=\exp(-\xi_{\tau(t)}), \quad \forall t \geq 0,
$$
where $\xi$ is a subordinator with Laplace exponent $$\phi(\lambda)=c+\int_{\mathcal S^{\downarrow}_-}\Big(1-\sum_{i\geq 1} s_i\Big) \nu(\mathrm d \mathbf s)+c\lambda+\int_{\mathcal S^{\downarrow}_-}\sum_{i\geq 1} (1-s_i^{\lambda})s_i \nu(\mathrm d \mathbf s), \quad \lambda \geq 0$$ 
and $\tau$ is a time--change depending on the parameter $\alpha$, $\tau(t)=\inf\left\{u\geq 0:\int_0^u \exp(\alpha \xi_t) \mathrm dr>t\right\}$. 
\end{thm}

\subsubsection{Self--similar fragmentation trees} 
\label{sec:ssft}

It was shown in \cite{HM04} that to every self--similar fragmentation process with a negative index $\alpha=-\gamma<0$, no erosion ($c=0$) and a dislocation measure $\nu$ satisfying  $\nu(\sum_{i \geq 1}s_i<1)=0$ (we say that $\nu$ is \emph{conservative}), there is an associated compact rooted measured tree that describes its genealogy. We denote such a tree by $(\mathcal T_{\gamma,\nu}, \mu_{\gamma,\nu})$ and precise that the measure $\mu_{\gamma,\nu}$ is fully supported by the set of leaves of $\mathcal T_{\gamma,\nu}$ and non--atomic. In \cite{Steph13}, Stephenson more generally constructed and studied  compact rooted measured trees that describe the genealogy of any self--similar fragmentation process with a negative index, however in this survey, we restrict ourselves to the family of trees $(\mathcal T_{\gamma,\nu}, \mu_{\gamma,\nu})$, with $\gamma>0$ and $\nu$ conservative. 

The connection between a tree $(\mathcal T_{\gamma,\nu}, \mu_{\gamma,\nu})$ and the fragmentation process it is related to can be summarized as follows: for all $t\geq 0$, consider the connected components of $\{v \in \mathcal T_{\gamma,\nu}: \mathrm{ht}(v)>t \}$, the set of points in $\mathcal T_{\gamma,\nu}$ that have a height strictly larger than $t$, and let $F(t)$ denote the decreasing rearrangement of the $\mu_{\gamma,\nu}$--masses of these components. Then $F$ is a fragmentation process, with index of self--similarly $-\gamma$, dislocation measure $\nu$ and no erosion. Besides, we note that $\mathcal T_{\gamma,\nu}$ possesses a fractal property, in the sense that if we fix a $t\geq 0$ (deterministic) and consider a point $x$ at height $t$, then any subtree of $\mathcal T_{\gamma,\nu}$ descending from this point $x$ (i.e. any connected component of $\{v\in \mathcal T_{\gamma,\nu}: x \in [[\rho,v]] \}$), having, say, a $\mu_{\gamma,\nu}$--mass $m$, is distributed as $m^{\gamma}\mathcal T_{\gamma,\nu}$.

\bigskip

\textbf{First examples.} The Brownian tree and the $\alpha$--stable trees that arise as scaling limits of Galton--Watson trees all belong to the family of self--similar fragmentation trees. More precisely,

$\bullet$ Bertoin \cite{BertoinSSF} notices that the Brownian tree ($\mathcal T_{\mathrm{Br}}$, $\mu_{\mathrm{Br}}$) is a self--similar fragmentation tree and calculates its characteristics: $\gamma=1/2$ and $\nu_{\mathrm{Br}}(s_1+s_2<1)=0$ and $$\nu_{\mathrm{Br}}(s_1 \in \mathrm dx)=\frac{\sqrt 2}{\sqrt \pi x^{3/2}(1-x)^{3/2}},  \quad1/2<x<1.$$
The fact that $\nu_{\mathrm{Br}}(s_1+s_2<1)=0$ corresponds to the fact the tree is binary: every branch point has two descendants trees, and in the corresponding fragmentation, every splitting events give two fragments.  

$\bullet$ Miermont \cite{Mier03} proves that each stable tree $\mathcal T_{\alpha}$ is self--similar and calculates its characteristics when $\alpha \in (1,2)$: $\gamma=1-1/\alpha$ and 
$$
\int_{\mathcal S^{\downarrow}}f(\mathbf s)\nu_{\alpha}(\mathrm d \mathbf s) =
C_{\alpha} \mathbb E\left[T_1 f\left(\frac{\Xi_i}{T_1},i \geq 1\right)\right],
$$
where 
\begin{equation*}
\label{Cbeta}
C_{\alpha}=\frac{\alpha(\alpha-1) \Gamma(1- 1/\alpha)}{\Gamma(2 -\alpha)}
\end{equation*}  
and  $\left(\Xi_i,i\geq 1 \right)$ is the sequence of lengths, ranked in the decreasing order, of intervals between successive atoms of a Poisson measure on $\mathbb R_+$ with intensity $(\alpha\Gamma(1-1/\alpha))^{-1} \mathrm dr /r^{1+1/\alpha}$, and $T_1=\sum_{i}\Xi_i$.

\bigskip

\textbf{Hausdorff dimension.} We first quickly recall the definition of Hausdorff dimension, which is a quantity that measures the ``size" of metric spaces. We refer to the book of Falconer \cite{Falc14} for more details on that topic and for an introduction to fractal geometry in general. For all $r>0$, the $r$--dimensional Hausdorff measure of a metric space $(Z,d_Z)$ is defined by
$$
\mathcal M^{r}(Z):=\lim_{\varepsilon \rightarrow 0} \inf_{\{(C_i)_{i\in \mathbb N} : \mathrm{diam}C_i\leq \varepsilon\}}\left\{\sum_{i \geq 1} \mathrm{diam}(C_i)^{r} : Z \subset \cup_{i\geq1} C_i \right\}.
$$
where the infimum is taken over all coverings of $Z$ by countable families of subsets $C_i \subset Z$, $i\in \mathbb N$, all with a diameter smaller than $\varepsilon$. The function $r >0 \mapsto \mathcal M^{r}(Z) \in [0,\infty]$ is finite, non--zero at most one point. 
The Hausdorff dimension of $Z$ is then given by
$$
\dim_{\mathrm{H}}(Z)=\inf\big\{r>0: \mathcal M^{r}(Z)=0\big\}=\sup\big\{r>0: \mathcal M^{r}(Z)=\infty\big\}.
$$

The Hausdorff dimension of a fragmentation tree depends mainly on its index of self--similarity. Let $\mathcal L(\mathcal T_{\gamma,\nu})$ denote the set of leaves of $\mathcal T_{\gamma,\nu}$. Then we know that, 
\begin{thm}[\cite{HM04}]
If $\int_{\mathcal S^{\downarrow}}(s_1^{-1}-1) \nu (\mathrm d \mathbf s)<\infty$, then almost surely
$$
\dim_{\mathrm H}(\mathcal L\left({\mathcal T_{\gamma,\nu}})\right)=\frac{1}{\gamma}
\quad\text{and} \quad
\dim_{\mathrm H}(\mathcal T_{\gamma,\nu})=\max\left(\frac{1}{\gamma},1\right). 
$$
\end{thm}
In particular, the Hausdorff dimension of the Brownian tree is 2, and more generally the Hausdorff dimension of the $\alpha$-stable tree, $\alpha \in (1,2)$, is $\alpha/(\alpha-1)$. This recovers a result of Duquesne and Le Gall \cite{DLG05}  proved in the framework of Lévy trees (we note that the intersection between the set of Lévy trees and that of self--similar fragmentation trees is exactly 
the set of stable Lévy trees).

\bigskip

\textbf{Height of a typical leaf.} The measured tree $(\mathcal T_{\gamma,\nu}, \mu_{\gamma,\nu})$ has been constructed in such a way that if we pick a leaf $L$ at random in $\mathcal T_{\gamma,\nu}$ according to $ \mu_{\gamma,\nu}$ and we consider for each $t\geq 0$ the $\mu_{\gamma,\nu}$-mass of the connected component of $\big\{v \in \mathcal T_{\gamma,\nu}: \mathrm{ht}(v)>t \big\}$ that contains this marked leaf (with the convention that this mass is 0 if $\mathrm{ht}(L)\leq t$), then we obtain a process which is distributed as the tagged fragment of the corresponding fragmentation process, as defined in the previous section. In particular, the height of $L$ is distributed as the absorption time of the process $\exp(-\xi_{\tau})$ introduced in Theorem \ref{thm:tagged}, i.e.
\begin{equation}
\label{Lawtypical}
\mathrm{ht}(L)\overset{\mathrm{(d)}}{=}\int_0^{\infty}\exp(-\gamma \xi_r) \mathrm dr
\end{equation}
where $\xi$ is a subordinator with Laplace exponent $\phi(\lambda)=\int_{\mathcal S^{\downarrow}}\sum_{i\geq 1} (1-s_i^{\lambda})s_i \nu(\mathrm d \mathbf s), \lambda \geq 0$. This is exactly the distribution of the limit appearing in Corollary \ref{corocvMC}.

\subsection{Scaling limits of Markov--Branching trees}
\label{sec:proof}

We can now explain the main steps of the proof of Theorem \ref{mainthm:leaves}. In that aim, let $(T_n,n\geq 1)$ denote a MB--sequence indexed by the leaves with splitting probabilities satisfying $(\mathsf H)$, with parameters $(\gamma,\nu)$ in the limit.
We actually only give here a hint of the proof of the convergence of the rescaled trees and refer to \cite[Section 4.4]{HM12} to see how to incorporate the measures. The proof of the convergence of the rescaled trees proceeds in three main steps:

\bigskip

\textbf{First step: convergence of the height of a typical leaf.} Keeping the notations previously introduced, $\mathrm{ht}(\star_n)$ for the height of a typical leaf in $T_n$ and $\mathrm{ht}(L)$ for the height of a typical leaf in a fragmentation tree $(\mathcal T_{\gamma,\nu}, \mu_{\gamma,\nu})$, we get by (\ref{Lawtypical}) and Corollary \ref{corocvMC} that
$$
\frac{\mathsf{ht}(\star_{n})}{n^{\gamma}} \overset{\mathrm{(d)}}{\underset{n \rightarrow \infty} \longrightarrow} \mathrm{ht}(L).
$$

\medskip

\textbf{Second step: convergence of finite--dimensional marginals.} For all integers $k \geq 2$, let $T_n(k)$ be the subtree of $T_n$ spanned by the root and $k$ (different) leaves picked independently, uniformly at random. Similarly, let $\mathcal T_{\gamma,\nu}(k)$  be the subtree of $\mathcal T_{\gamma,\nu}$ spanned by the root and $k$ leaves picked independently at random according to the measure $\mu_{\gamma,\nu}$. Then (under $(\mathsf H)$),
\begin{equation}
\label{fdcv}
\frac{T_n(k)}{n^{\gamma}} \overset{\mathrm{(d)}}{\underset{n \rightarrow \infty}\longrightarrow} \mathcal T_{\gamma,\nu}(k).
\end{equation}
This can be proved by induction on $k$. For $k=1$, this is Step 1 above. For $k \geq 2$, we use the induction hypothesis and the MB--property. Here is the main idea. Consider the decomposition of $T_n$ into subtrees above its first branch point in $T_n(k)$ and take only into account the subtrees having marked leaves. We obtain $m \geq 2$ subtrees with, say, $n_1,\ldots,n_m$ leaves respectively ($\sum_{i=1}^m n_i \leq n$), and each of these trees have $k_1\geq 1,\ldots k_m\geq 1$ marked leaves ($\sum_{i=1}^m k_i =k$). Given $m$,  $n_1,\ldots,n_m$, $k_1,\ldots k_m$, the MB--property ensures that the $m$ subtrees are independent with respective distributions that of $T_{n_1}(k_1),\ldots,T_{n_m}(k_m)$. An application of the induction hypothesis to these subtrees leads to the expected result. We refer to \cite[Section 4.2]{HM12} for details. 

\bigskip

\textbf{Third step: a tightness criterion.} To get the convergence for the GH--topology, the previous result must be completed with a tightness criterion. The idea is to use the following well--known result.
\begin{thm}[\cite{Bill99}, Theorem 3.2]
If $X_n,X,X_n(k),X(k)$ are r.v. in a metric space $(E,d)$ such that
$
X_n(k) \overset{\mathrm{(d)}}{\underset{n \rightarrow \infty}\longrightarrow} X(k)$, $\forall k$
and 
$
X(k) \overset{\mathrm{(d)}}{\underset{k \rightarrow \infty}\longrightarrow}X
$
and for all $\varepsilon>0$,
\begin{equation}
\label{tightnesscriterion}
\lim_{k \rightarrow \infty} \limsup_{n \rightarrow \infty} \mathbb P\left(d(X_n,X_n(k))>\varepsilon \right)=0
\end{equation}
then 
$X_n \overset{\mathrm{(d)}}{\underset{n \rightarrow \infty}\longrightarrow} X$.
\end{thm}
In our context, the finite--dimensional convergence (\ref{fdcv}) has already been checked. Moreover, since $\mu_{\gamma,\nu}$ is fully supported on the set of leaves of $\mathcal T_{\gamma,\nu}$, we see by picking an infinite sequence of i.i.d. leaves according to $\mu_{\gamma,\nu}$, that there exist versions of the  $\mathcal T_{\gamma,\nu}(k),k \geq 1$ that converge almost surely to $\mathcal T_{\gamma,\nu}$ as $k \rightarrow \infty$. It remains to  
establish the tightness criterion (\ref{tightnesscriterion}) for $T_n,T_n(k)$, with respect to the distance $d_{\mathrm{GH}}$. The main tool is the following bounds:
\begin{prop}
Under $(\mathrm{H})$, for all $p>0$, there exists a finite constant $C_p$ such that
$$
\mathbb P \left(\frac{\mathsf{ht}(T_n)}{n^{\gamma}} \geq x \right) \leq \frac{C_p}{x^p}, \quad \forall x>0, \forall n \geq 1. 
$$
\end{prop}
The proof holds by induction on $n$, using $(\mathrm{H})$ and the MB--property. We refer to \cite[Section 4.3]{HM12} for details and to see how, using again the MB--property, this helps to control the distance between $T_n$ and $T_n(k)$, to get that for $\varepsilon>0$:
$$
\lim_{k \rightarrow \infty} \limsup_{n \rightarrow \infty} \mathbb P\left(d_{\mathrm{GH}} \left(\frac{T_n(k)}{n^{\gamma}},\frac{T_n}{n^{\gamma}}\right) \geq \varepsilon\right)=0
$$ 
as required.

\section{Applications}
\label{sec:applic}

We now turn to the description of the scaling limits of various models of random trees that are closely linked to the MB--property.

\subsection{Galton--Watson trees}

\subsubsection{Galton--Watson trees with $n$ vertices}

A first application of Theorem \ref{mainthm:vertices} is that it permits to recover the classical results of Aldous and Duquesne (grouped together in Theorem \ref{thm:AldousDuquesne}) on the scaling limits of Galton--Watson trees conditioned to have $n$ vertices. To see this, one just has to check the two following lemmas, for $\eta$ a critical offspring distribution, $\eta(1)\neq 1$, and $(p_n^{\mathrm{GW},\eta})$ the associated splitting distributions defined in (\ref{splitGW}).   

\begin{lem}
\label{lemGW1}
If $\eta$ has a finite variance $\sigma^2$, then $(p_n^{\mathrm{GW},\eta})$ satisfies $\mathrm{(H)}$ with 
$$\gamma=1/2 \quad \text{and} \quad  \nu=\frac{\sigma}{2} \nu_{\mathrm Br}.$$
\end{lem}

\bigskip

\begin{lem}
\label{lemGW2}
If $\eta(k) \sim \kappa k^{-\alpha-1} $ for some $\alpha \in (1,2)$, then $(p_n^{\mathrm{GW},\eta})$  satisfies $\mathrm{(H)}$ with 
$$\gamma=1-1/\alpha \quad and \quad \nu=\left(\kappa \Gamma(2-\alpha)\alpha^{-1}(\alpha-1)^{-1} \right)^{1/\alpha}\nu_{\alpha}.$$
\end{lem}

The measures $\nu_{\mathrm{Br}}$ and $\nu_{\alpha}$ are the dislocation measures of the Brownian and $\alpha$--stable tree, respectively, and are defined in Section \ref{sec:ssft}.
Together with the scaling limit results on MB--trees, this gives Theorem \ref{thm:AldousDuquesne} (i) and (ii) respectively. 
The proofs of these lemmas are not completely obvious. We give here a rough idea of the main steps of the proof of Lemma \ref{lemGW1}, and refer to \cite[Section 5]{HM12} for more details and for the proof of Lemma \ref{lemGW2}.

\textbf{Sketch of the main steps of the proof of Lemma \ref{lemGW1}.} Recall that $T^{\eta}$ denotes a Galton--Watson tree with offspring distribution $\eta$. To simplify, we assume that the support of $\eta$ generates $\mathbb Z$, so that $\mathbb P(\#_{\mathrm{vertices}} T^{\eta}=n)>0$ for all $n$ large enough.  The Otter--Dwass formula (or cyclic lemma) \cite[Chapter 6]{PitmanStFl} then implies that
$$
\mathbb P(\#_{\mathrm{vertices}}  T^{\eta}=n)=n^{-1} \mathbb P(S_n=-1)
$$
where $S_n$ is a random walk with i.i.d. increments of law $(\eta(i+1), i \geq -1)$. Together with the local limit Theorem, which ensures that $\mathbb P(S_n=-1)\underset{n \rightarrow \infty}\sim (2 \pi \sigma^2 n)^{-1/2}$, this leads to
$$
\mathbb P(\#_{\mathrm{vertices}}  T^{\eta}=n)\underset{n \rightarrow \infty}\sim (2 \pi \sigma^2)^{-1/2}n^{-3/2}.
$$
(We note that this argument is also fundamental in the study of large Galton--Watson trees via their contour functions).
Then the idea is to use this approximation in the definition of $p_n^{\mathrm{GW},\eta}$ to show that the two sums 
\begin{equation*}
\sqrt n \sum_{\lambda \in \mathcal P_{n}}p_{n}^{\mathrm{GW},\eta}(\lambda)\left(1-\frac{\lambda_1}{n}\right) f\left(\frac{\lambda}{n}\right) \quad \text{and} \quad 
 \frac{\sigma}{ \sqrt{2\pi}} \frac{1}{n} \sum_{\lambda_1 = \lceil n/2 \rceil}^{n} f\left(\frac{\lambda_1}{n},\frac{n-\lambda_1}{n},\ldots\right)\left(\frac{\lambda_1}{n}\right)^{-3/2} \left(\frac{n-\lambda_1}{n}\right)^{-3/2}
\end{equation*}
are asymptotically equivalent (this is the technical part), for all continuous functions $f:\mathcal S^{\downarrow} \rightarrow \mathbb R.$ 
The conclusion follows, since the second sum is a Riemann sum that converges to \linebreak $\sigma(\sqrt{2\pi})^{-1} \int_{1/2}^{1} f(x,1-x,\ldots)x^{-3/2}(1-x)^{-3/2} \mathrm dx.$

\subsubsection{Galton--Watson trees with arbitrary degree constraints}
\label{GWmore}

One may then naturally wonder if Theorem \ref{mainthm:leaves} could also be used to get the scaling limits of Galton--Watson trees conditioned to have $n$ leaves. The answer is yes, and moreover this can be done in a larger context, using a simple generalization of Theorem \ref{mainthm:leaves} and Theorem \ref{mainthm:vertices} to MB--trees with arbitrary degree constraints. This generalization was done by Rizzolo  \cite{Riz15} using an idea similar to the one presented below  Theorem \ref{mainthm:vertices} to get this theorem from Theorem \ref{mainthm:leaves}. It is quite heavy to state neatly, so we let the reader see the paper  \cite{Riz15} and focus here on the applications developed in this paper to Galton--Watson trees.

The classical theorems of Aldous and Duquesne on conditioned Galton--Watson trees can be extended to Galton--Watson trees conditioned to have a number of vertices with out--degree in a given set. To be more precise, fix $A \subset \mathbb Z_+$ and consider an offspring distribution $\eta$ with mean 1 and variance $0<\sigma^2<\infty$.  For integers $n$ for which such a conditioning is possible, let 
$T_n^{\eta,A}$ denote a version of a  $\eta-$Galton--Watson tree conditioned to have $n$ vertices with out--degree in $A$. For example, if $A=\mathbb Z_+$, this is the model of the previous section, whereas if $A=\{0\}$,  $T_n^{\eta,A}$ is a $\eta-$Galton--Watson tree conditioned to have $n$ leaves.

\begin{thm}[Kortchemski \cite{Kort12} and Rizzolo \cite{Riz15}]
As $n \rightarrow \infty$,
$$\left(\frac{T_n^{\eta,A}}{\sqrt n},\mu_n^{\eta,A}\right) \overset{\mathrm{(d)}}{\underset{\mathrm{GHP}} \longrightarrow} \left(\frac{2}{\sigma \sqrt{\eta(A)}} \mathcal T_{\mathrm{Br}},\mu_{\mathrm{Br}}\right).$$
\end{thm}

The proof of Rizzolo \cite{Riz15} relies on his theorem on scaling limits of MB--trees with arbitrary degree constraints. The most technical part is to evaluate the splitting probabilities, which is done by generalizing the Otter--Dwass formula. The proof of Kortchemski \cite{Kort12} is more in the spirit of the proofs of Aldous and Duquesne and consists in studying the contour functions of the conditioned trees. We note that \cite{Kort12} also includes cases where $\eta$ has an infinite variance, and is in the domain of attraction of a stable distribution (the limit is then a multiple of a stable tree). It should be possible to recover this more general case via the approach of Rizzolo.

\textbf{An example of application to combinatorial trees indexed by the number of leaves.} Let $T_n$ be a tree uniformly distributed amongst the set of rooted ordered trees with $n$ leaves with no vertex with out--degree 1.  One checks, using (\ref{lawGW}),  that $T_n$ is distributed as a $\eta$--Galton--Watson tree conditioned to have $n$ leaves, with $\eta$ defined by $\eta(i)=(1-2^{-1/2})^{i-1}, i \geq 2$, $\eta(1)=0$ and $\eta(0)=2-2^{1/2}$. The variance of $\eta$ is $4 (\sqrt 2- 1)$, so that finally, 
$$
\left(\frac{T_n}{\sqrt n},\mu_n\right) \overset{\mathrm{(d)}}{\underset{\mathrm{GHP}} \longrightarrow}\left(\frac{1}{ 2^{1/4}(\sqrt 2 -1)} \mathcal T_{\mathrm{Br}},\mu_{\mathrm{Br}}\right)
$$
where $\mu_n$ is the uniform probability on the leaves of $T_n$.

\subsection{P\'olya trees}
\label{sec:Polya}

The above results on conditioned Galton--Watson trees give the scaling limits of several sequences of combinatorial trees, as already mentioned. There is however a significant case which does not fall within the Galton--Watson framework, that of \emph{uniform P\'olya trees}. By P\'olya trees we simply mean rooted finite trees (non--ordered, non--labelled).  They are named after P\'olya \cite{Polya37} who developed an analytical treatment of this family of trees, based on generating functions. In this section, we let $T_n(\mathrm P)$ be uniformly distributed amongst the set of P\'olya trees with $n$ vertices. 

These trees are more complicated to study than uniform rooted trees with labelled vertices, or  uniform rooted, ordered trees, because of their lack of symmetry.  In this direction, Drmota and Gittenberger \cite{DrG10} showed that the shape of $T_n(\mathrm P)$ \emph{is not} a conditioned Galton--Watson tree. 
However Aldous \cite{Ald91} conjectured in 1991 that  the scaling limit of $(T_n(\mathrm P))$ should nevertheless be the Brownian tree, up to a multiplicative constant. 
Quite recently, several papers studied the scaling limits of P\'olya trees, with different points of view. Using techniques of analytic combinatorics,  Broutin and Flajolet \cite{BrF08} studied the scaling limit of the height of a uniform \emph{binary} P\'olya tree with $n$ vertices, whereas Drmota and Gittenberger \cite{DrG10} studied the profil of $T_n(\mathrm P)$  (the profile is the sequence of the sizes of each generation of the tree) and showed that it converges after an appropriate rescaling to the local time of a Brownian excursion. Marckert and Miermont  \cite{MM11} obtained a full scaling limit picture of uniform \emph{binary} P\'olya trees: by appropriate trimming procedures, they showed that rescaled by $\sqrt n$, they converge in distribution towards a multiple of the Brownian tree.  

More recently, with different methods, the following result was proved.

\begin{thm}[Haas--Miermont \cite{HM12} and Panagiotou--Stufler \cite{PS16+}]
\label{thmPolya}
As $n \rightarrow \infty$,
$$\left(\frac{T_n(\mathrm P)}{\sqrt n},\mu_n(\mathrm P)\right) \overset{\mathrm{(d)}}{\underset{\mathrm{GHP}} \longrightarrow}  \big(c_{\mathrm P}\mathcal T_{\mathrm{Br}}, \mu_{\mathrm{Br}}\big), \quad c_{\mathrm P} \sim 1.491$$
where $\mu_n(\mathrm P)$ denotes the uniform probability on the vertices of $T_n(\mathrm P)$.
\end{thm}

The proof of \cite{HM12} uses connections with MB--trees, whereas that of \cite{PS16+} uses, still, connections with Galton--Watson trees. Let us first quickly discuss the MB point of view. 
It is easy to check that the sequence ($T_n(\mathrm P)$) is not Markov--Branching (this is left as an exercise!), however it is not far from being so. It is actually possible to couple this sequence with a Markov--Branching sequence $(T'_n(\mathrm P))$ such that 
$$\mathbb E\left[d_{\mathrm{GHP}}(n^{-\varepsilon} T_n(\mathrm P),n^{-\varepsilon} T'_n(\mathrm P))\right] \underset{n \rightarrow \infty}\longrightarrow 0, \quad \forall \varepsilon>0$$
and $(T_n(\mathrm P))$ and $(T'_n(\mathrm P))$ have the same splitting probabilities $(p_n)$ (by splitting probabilities for trees that are not MB, we mean the distribution of the sizes of the subtrees above the root). These splitting probabilities are given here by 
$$
p_{n-1}(\lambda)=\frac{\prod_{j=1}^{n-1} \# F_j(m_j(\lambda))}{ \# \mathbb T_n}, \quad \text{for } \lambda\in \mathcal P_{n-1}
$$
where $m_j(\lambda)=\{i:\lambda_i=j\}$, $\# \mathbb T_n$ is the number of rooted trees with $n$ vertices
and $F_j(k)$ denotes the set of multisets with $k$ elements in $\mathbb T_j$ (with the convention $F_j(0):=\{\emptyset\}$). It remains to check that these splitting probabilities satisfy $(\mathsf H)$ with appropriate parameters and to do this, we use the result of Otter \cite{Otter48}:
$$\# \mathbb T_n \underset{n \rightarrow \infty}\sim \mathrm{c} \frac{\kappa^{n}}{n^{3/2}}, \quad \text{for some }\mathrm{c}>0, \kappa>1.$$
Very roughly, this allows to conclude that the two following sums 
$$\sqrt n \sum_{\lambda \in \mathcal P_{n}}p_n(\lambda)\left(1-\frac{\lambda_1}{n}\right) f\left(\frac{\lambda}{n}\right) \quad \text{and} \quad \frac{\mathrm c}{n} \sum_{\lambda_1 = \lceil (n-1)/2 \rceil}^{n-1} f\left(\frac{\lambda_1}{n},\frac{n-\lambda_1}{n},0,\ldots \right)\left(\frac{\lambda_1}{n}\right)^{-3/2} \left(\frac{n-\lambda_1}{n}\right)^{-3/2}$$
are asymptotically equivalent, 
so that finally, using that the second sum is a Riemann sum, $(\mathsf H)$ holds with parameters $\gamma=1/2$ and $\nu=\nu_{\mathrm{Br}}/c_P$, with $c_{\mathrm P}=\sqrt 2/(\mathrm{c} \sqrt{\pi})$. The method of \cite{PS16+} is different. It consists in showing that asymptotically $T_n(\mathrm P)$ can be seen as  a large finite--variance critical Galton--Watson tree of random size concentrated around a constant times $n$ on which small subtrees of size $O(\log(n))$ are attached. The conclusion then follows from the classical result by Aldous on scaling limits of Galton--Watson trees.

Both methods can be adapted to P\'olya trees with other degree constraints. In \cite{HM12} uniform P\'olya trees with $n$ vertices having out--degree in $\{0,m\}$ for some fixed integer $m$, or out--degree at most $m$, are considered.  More generally, in  \cite{PS16+},  uniform P\'olya trees with $n$ vertices having out--degree in a fixed set $A$ (containing at least 0 and an integer larger than 2) are studied. In all cases, the trees rescaled by $\sqrt n$ converge in distribution towards a multiple of the Brownian tree.

\textbf{Further result.} To complete the picture on combinatorial trees asymptotics, we mention a recent result by Stufler on \emph{unrooted} trees, that was conjectured by Aldous, but remained open for a while. 

\begin{thm}[Stufler \cite{Stu16+}]
Let $T_n^*(\mathrm P)$ be uniform amongst the set of unrooted trees with $n$ vertices $($unordered, unlabelled$)$. Then,
$$\frac{T_n^*(\mathrm P)}{\sqrt n} \overset{\mathrm{(d)}}{\underset{\mathrm{GH}}\longrightarrow}  c_{\mathrm P}\mathcal T_{\mathrm{Br}}$$
$($with the same $c_{\mathrm P}$ as in Theorem \ref{thmPolya}$)$.
\end{thm}

The main idea to  prove this scaling limit of unrooted uniform trees consists in using a decomposition due to Bodirsky, Fusy, Kang and Vigerske \cite{BFKV11} to approximate $T_n^*(\mathrm P)$ by uniform \emph{rooted} P\'olya trees and then use Theorem \ref{thmPolya}. This result more generally holds for unrooted trees with very general degree constraints.

\subsection{Dynamical models of tree growth}
\label{sec:cvdynam}

As mentioned in Section \ref{sec:examples}, the prototype example of Rémy's algorithm $(T_n(\mathrm R),n \geq 1)$ is strongly connected to Galton--Watson trees since the shape of $T_n(\mathrm R)$ (to which has been subtracted the edge between the root and the first branch point) is distributed as the shape of a binary critical Galton--Watson tree conditioned to have $2n-1$ vertices. This implies that
$$
\left(\frac{T_n(\mathrm R)}{\sqrt n},\mu_n(\mathrm R)\right) \overset{\mathrm{(d)}}{\underset{\mathrm{GHP}}\longrightarrow} \left(2\sqrt 2 \mathcal T_{\mathrm{Br}}, \mu_{\mathrm{Br}}\right).
$$ 
Similar scaling limits results actually extends to most of the tree--growth models seen in Section \ref{sec:examples}. To see this, it suffices to check that their splitting probabilities satisfy Hypothesis $(\mathsf H)$. Technically, this mainly relies on Stirling's formula and/or balls in urns schemes. Note however that the convergence in distribution is not fully satisfactory in these cases, since the trees are recursively built on a same probability space, and we may hope to have convergence in a stronger sense. We will see below that this is indeed the case. 

\subsubsection{Ford's alpha model}

For Ford's $\alpha$-model, with $\alpha \in (0,1)$, it is easy to check (see \cite{HMPW08}) that the splitting probabilities $q_n^{\mathrm{Ford},\alpha}$ satisfy hypothesis $(\mathsf H)$ with $\gamma=\alpha$ and $\nu=\nu_{\mathrm{Ford},\alpha}$,  
where $\nu_{\mathrm{Ford},\alpha}$ is a binary measure on $\mathcal S^{\downarrow}$ \linebreak  ($\nu_{\mathrm{Ford},\alpha }(s_1+s_2<1)=0$) defined by
\vspace{-0.10cm}
$$
\nu_{\mathrm{Ford},\alpha}(s_1 \in \mathrm dx)=\frac{ \mathbbm1_{\{1/2 \leq x \leq 1\}}}{\Gamma(1-\alpha )} \left(\alpha (x(1-x))^{-\alpha -1} + (2-4 \alpha )(x(1-x))^{-\alpha } \right)  \mathrm dx.
$$
This, together with Theorem \ref{mainthm:vertices} leads for $\alpha \in (0,1)$ to the convergence:
\begin{thm}[\cite{HMPW08} and \cite{HM12}] For all $\alpha  \in (0,1)$, 
$$
\left(\frac{T_n(\alpha )}{n^\alpha },\mu_n(\alpha)\right) \overset{\mathrm{(d)}}{\underset{\mathrm{GHP}}\longrightarrow} \left(\mathcal T_{\alpha ,\nu_{\mathrm{Ford},\alpha }}, \mu_{\alpha,\nu_{\mathrm{Ford},\alpha }}\right).$$
\end{thm}

This result was actually first proved in \cite{HMPW08}, using the fact that the sequence $(T_n(\alpha))$ is Markov--Branching and consistant. Chen and Winkel \cite{CW13} then  improved this result by showing that the convergence holds in probability. 

For $\alpha=1/2$ (Rémy's algorithm), note that we recover the result obtained via the Galton--Watson approach. Note also that the case $\alpha=1$ is not included in the hypotheses of the above theorem, however the trees $T_n(\alpha)$ are then deterministic (comb trees) and it is clear that they converge after rescaling by $n$ to a segment of length 1, equipped with the Lebesgue measure. This tree is a \emph{general fragmentation tree} as introduced by Stephenson \cite{Steph13}, with pure erosion (and no dislocation). 
When $\alpha=0$, we observe a different regime, the height of a typical leaf in the tree growth logarithmically, and there is no convergence in the GH--sense of the whole tree.  

\subsubsection{$k$--ary growing trees}
\label{scalkary}

Observing the asymptotic behavior of the sequence of trees constructed via Rémy's algorithm, it is natural to wonder how this may change when deciding to branch at each step $k-1$ branches on the pre--existing tree, instead of one. For this $k-$ary model, it was shown in \cite{HS15} that $q_n^k$ satisfies $(\mathrm H)$ with $\gamma=1/k$ and $\nu=\nu_k$ where 
$$
\nu_k({\mathrm d \mathbf s})=\frac{(k-1)!}{k(\Gamma(\frac{1}{k}))^{k-1}} \prod_{i=1}^k s_i^{-(1-1/k)}\left(\sum_{i=1}^k \frac{1}{1-s_i}\right)\mathbbm 1_{\{s_1\geq s_2 \geq ... \geq s_k\}}\mathrm d \mathbf s,
$$  
is supported on the simplex of dimension $k-1$.
Together with Theorem  \ref{mainthm:leaves} this gives the limit in distribution of the sequence $(T_n(k),n \geq 1)$. 
Besides, using some connections with the Chinese Restaurant Processes of Dubins and Pitman (see \cite[Chapter3]{PitmanStFl} for a definition) and more general urns schemes, it was shown that these models converge in probability (however this second approach did not give the distribution of the limiting tree.) Together, these two methods lead to:
\begin{thm}[\cite{HS15}]
\label{thm:kary}
Let $\mu_n(k)$ be the uniform measure on the leaves of $T_n(k)$.
Then,
$$
\left(\frac{T_n(k)}{n^{1/k}}, \mu_n(k) \right)\ \overset{\mathbb{P}}{\underset{\mathrm{GHP}}\longrightarrow} \ \left(\mathcal T_{k},\mu_k\right)
$$
where  \emph{($\mathcal T_k, \mu_k$)} is a self--similar fragmentation tree, with index of self--similarity $1/k$ and dislocation measure $\nu_k$.\end{thm}

Interestingly, using the approximation by discrete trees, it is possible to show that randomized versions of the limiting trees $\mathcal T_{k},k\geq 2$ -- note that $\mathcal T_2$ is the Brownian tree up to a scaling factor -- can be embedded into each other so as to form an increasing (in $k$) sequence of trees \cite[Section 5]{HS15}.

In Section \ref{sec:multitype}  we will discuss a generalization of this model. Here, we glue at each step star--trees with $k-1$ branches. However more general results are available when deciding to glue more general tree structures, with possibly a random number of leaves.    

\subsubsection{Marginals of stable trees}
\label{sec:marginals}

For $\beta \in (1,2]$, the sequence $(T_n(\beta),n\geq 1)$ built by Marchal's algorithm provides, for each $n$, a tree that is distributed as the shape of the subtree of the stable tree $\mathcal T_{\beta}$ spanned by $n$ leaves  taken independently according to $\mu_{\beta}$. Duquesne and Le Gall \cite{DLG02} showed that $T_n(\beta)$ is distributed as a Galton--Watson tree whose offspring distribution has probability generating function $z+\beta^{-1}(1-z)^{\beta}$, conditioned to have $n$ leaves. As so, it is not surprising that appropriately rescaled it should converge to the $\beta$--stable tree. Marchal \cite{Mar06} proved an almost--sure \emph{finite--dimensional convergence}, whereas the results of \cite{HMPW08} give the convergence in probability for the GHP--topology.
Additional manipulations even lead to an almost--sure convergence for the GHP--topology:
\begin{thm}[\cite{CH13}]
Let $\mu_n(\beta)$ be the uniform measure on the leaves of $T_n(\beta)$. Then
$$
\left(\frac{T_n(\beta)}{n^{\beta}},\mu_n(\beta)\right) \ \overset{\mathrm{a.s.}}{\underset{\mathrm{GHP}}\longrightarrow} \ \left(\beta \mathcal T_{\beta},\mu_\beta\right)
$$
\end{thm}

Using this convergence, it was shown in \cite{CH13} that randomized versions of the stable trees $\mathcal T_{\beta}, 1 <\beta \leq 2$ can be embedded into each other so as to form a decreasing (in $\beta$) sequence of trees.

To complete these results, we mention that  Chen, Ford and Winkel \cite{CFW09} propose a model that interpolate between the $\alpha$--model of Ford and  Marchal's recursive construction of the marginals of stable trees, and determine there scaling limits, relying on the results of \cite{HMPW08}.

\subsection{Cut--trees}

The notion of the cut--trees was introduced in Example 5 of Section \ref{sec:MB}.
 
\textbf{Cut--tree of a uniform Cayley tree.} We use the notation of Example 5, Section \ref{sec:MB} and let $C_n$ be a uniform Cayley tree and $T^{\mathrm{cut}}_n$ its cut--tree. Relying essentially on Stirling's formula, one gets that $q_n^{\mathrm{Cut, Cayley}}$ satisfies $(\mathrm{H})$ with $\gamma=1/2$ and $\nu=\nu_{\mathrm{Br}}/2$, which shows that the rescaled cut--tree $T^{\mathrm{cut}}_n/\sqrt n$ endowed with the uniform measure on its leaves converges in distribution to $(2\mathcal T_{\mathrm{Br}},\mu_{\mathrm{Br}})$. This was noticed in \cite{BerFire} and used to determine the scaling limits of the number of steps needed to isolated by edges delation a fixed number of vertices in $C_n$. Actually, Bertoin and Miermont \cite{BM13} improve this result by showing the joint convergence

\begin{thm}[Bertoin--Miermont \cite{BM13}]
$$\left(\frac{C_n}{\sqrt n},\frac{T^{\mathrm{cut}}_n}{\sqrt n} \right)  \overset{\mathrm{(d)}}{\underset{\mathrm{GHP}}\longrightarrow}  \ \left(2\mathcal T_{\mathrm{Br}}, 2 \overline{\mathcal T_{\mathrm{Br}}} \right)$$
where $\overline{\mathcal T_{\mathrm{Br}}}$ is a tree constructed from $ \mathcal T_{\mathrm{Br}}$, that can be interpreted as its cut--tree, and that is distributed as $\mathcal T_{\mathrm{Br}}$. 
\end{thm}
Bertoin and Miermont \cite{BM13} actually more  generally extend this result to cut--trees of Galton--Watson trees with a critical offspring distribution with finite variance.
This in turn was generalized by Dieuleveut \cite{Dieul15} to Galton--Watson trees with a critical offspring distribution in the domain of attraction of a stable law. See also \cite{ABBrH14, BrW16, Brw16+} for related results. 

\textbf{Cut--tree of a uniform recursive tree.} On the other hand, note that $q_n^{\mathrm{Cut, Recursive}}$ does not satisfy $(\mathrm{H})$. However, Bertoin showed in \cite{Ber15} that in this case, the cut--tree $T_n$ rescaled by $n/\ln(n)$ converges for the GHP--topology to a segment of length 1, equipped with the Lebesgue measure.

\section{Further perspectives}
\label{sec:further}

\subsection{Multi--type Markov--Branching trees and applications}
\label{sec:multitype}

It is possible to enrich trees with \emph{types}, by deciding that each vertex of a tree carries a type, which is an element of a finite or countable set. This multi--type setting is often used in the context of branching processes, where individuals with different types may evolve differently, and had been widely studied. For the trees point of view, scaling limits of multi--type Galton--Watson trees conditioned to have a given number of vertices have been studied by Miermont \cite{Mier08} when both the set of types  and the covariance matrix of the offspring distributions are finite, by Berzunza \cite{Berz16+} when the set of types is finite  with offspring distributions in the domain of attraction of a stable distribution and by de Raphélis \cite{DeRaph15+} when the number of type is infinite, under a finite variance--type assumption. The Brownian and stable trees appear in the scaling limits.

One may more generally be interested in multi--type Markov--Branching trees, which are sequences of trees with vertices carrying types, where, roughly, the subtrees  above the root are independent and with distributions that only depend on their size and on the type of their root. In a work in progress \cite{HS15+},  results similar to Theorem \ref{mainthm:leaves} and Theorem \ref{mainthm:vertices} are set up for multi--type MB--trees, when the set of types is finite. Interestingly, different regimes appear in the scaling limits (multi--type or standard fragmentation trees), according to whether the rate of type change is faster or equal or slower than the rate of macroscopic branchings.

This should lead  to new proofs of the results obtained in \cite{Mier08,Berz16+}. This should also lead to other interesting applications, in particular to dynamical models of tree growth. In these growing models one starts from a finite alphabet of trees and then glues recursively trees by choosing at each step one tree at random in the alphabet and grafting it uniformly on an edge of the pre--existing tree. This generalizes the $k$--ary construction studied in Section \ref{sec:examples} and Section \ref{scalkary}, and is connected to multi--type MB--trees. In this general setting multi--type fragmentation trees will appear in the scaling limits.

\subsection{Local limits}

This survey deals with scaling limits of random trees. There is another classical way to consider limits of sequences of trees (or graphs), that of \emph{local limits}. This approach is quite different and provides other information on the asymptotics of the trees (e.g. on the limiting behavior of the degrees of vertices). Roughly, a sequence of finite rooted trees $(t_n)$ is said to converge locally to a limit $t$ if for all $R>0$, the restriction of $t_n$ to a ball of radius $R$ centered at the root converges to the restriction of $t$ to a ball of radius $R$ centered at the root. The trees are therefore not rescaled and the limit is still a discrete object.

For results on the local limits of random models related to the ones considered here, we refer to: Abraham and Delmas \cite {ADLoc2, ADLoc1} and the references therein for Galton--Watson trees, Stef\'ansson \cite{Ste09} for Ford's $\alpha$--model and Pagnard \cite{Pagnard16} for general MB--sequences and the study of the volume growth of their local limits. We also mention the related work by Broutin and Mailler \cite{BM15+} that uses local limits of some models of MB--trees to study asymptotics of And/Or trees, that code boolean functions.

\subsection{Related random geometric structures}

The discrete trees form a subclass of graphs and are generally simpler to study. There exist however several models of graphs (that are not trees) whose asymptotic study can be conducted by using trees. Different approaches are possible and it is not our purpose to present them here. However we still give some references that are related to some models of trees presented here (in particular Galton--Watson trees) to the interested reader   (the list is not exhaustive):
\begin{enumerate}
\item[$\bullet$] on random graphs converging to the Brownian tree: \cite{AlbMark08,Bett15, Cara15+,CHK14,JanStef15,PSW15+, Stu16+2}
\item[$\bullet$] on random graphs converging to tree--like structures:  \cite{CRLoop13,CKPerco,CDKMPref}
\item[$\bullet$] on the Erd\H{o}s--Rényi random graph in the critical window and application to the minimum spanning tree of the
complete graph: \cite{ABBG12,ABBGM13}
\item[$\bullet$] on random maps (which are strongly connected to labeled trees): \cite{MiermontStFlour} and all the references therein.
\end{enumerate} 
In most of these works the Brownian tree intervenes in the construction of the continuous limit.


%
%
%
%
%
\bibliographystyle{siam}
\bibliography{FragAvril16}

\begin{thebibliography}{10}

\bibitem{ADLoc2}
{\sc R.~Abraham and J.-F. Delmas}, {\em Local limits of conditioned
  {G}alton-{W}atson trees: the condensation case}, Electron. J. Probab., 19
  (2014), pp.~1--29.

\bibitem{ADLoc1}
\leavevmode\vrule height 2pt depth -1.6pt width 23pt, {\em Local limits of
  conditioned {G}alton-{W}atson trees: the infinite spine case}, Electron. J.
  Probab., 19 (2014), pp.~1--19.

\bibitem{ADH}
{\sc R.~Abraham, J.-F. Delmas, and P.~Hoscheit}, {\em A note on the
  {G}romov-{H}ausdorff-{P}rokhorov distance between (locally) compact metric
  measure spaces}, Electron. J. Probab., 18(14) (2013), pp.~1--21.

\bibitem{ABBG12}
{\sc L.~Addario-Berry, N.~Broutin, and C.~Goldschmidt}, {\em The continuum
  limit of critical random graphs}, Probab. Theory Related Fields, 152 (2012),
  pp.~367--406.

\bibitem{ABBGM13}
{\sc L.~Addario-Berry, N.~Broutin, C.~Goldschmidt, and G.~Miermont}, {\em The
  scaling limit of the minimum spanning tree of the complete graph},  (2013).
\newblock Preprint -- arXiv:1301.1664.

\bibitem{ABBrH14}
{\sc L.~Addario-Berry, N.~Broutin, and C.~Holmgren}, {\em Cutting down trees
  with a {M}arkov chainsaw}, Ann. Appl. Probab., 24 (2014), pp.~2297--2339.

\bibitem{AlbMark08}
{\sc M.~Albenque and J.-F. Marckert}, {\em Some families of increasing planar
  maps}, Electron. J. Probab., 13 (2008), pp.~no. 56, 1624--1671.

\bibitem{Ald91a}
{\sc D.~Aldous}, {\em The continuum random tree. {I}}, Ann. Probab., 19 (1991),
  pp.~1--28.

\bibitem{Ald91}
\leavevmode\vrule height 2pt depth -1.6pt width 23pt, {\em The continuum random
  tree. {II}. {A}n overview}, in Stochastic analysis ({D}urham, 1990), vol.~167
  of London Math. Soc. Lecture Note Ser., Cambridge Univ. Press, Cambridge,
  1991, pp.~23--70.

\bibitem{Ald93}
{\sc D.~Aldous}, {\em The continuum random tree {III}}, Ann. Probab., 21
  (1993), pp.~248--289.

\bibitem{Ald96}
{\sc D.~Aldous}, {\em Probability distributions on cladograms}, in Random
  discrete structures ({M}inneapolis, {MN}, 1993), vol.~76 of IMA Vol. Math.
  Appl., Springer, New York, 1996, pp.~1--18.

\bibitem{BertoinLevy}
{\sc J.~Bertoin}, {\em L\'evy processes}, vol.~121 of Cambridge Tracts in
  Mathematics, Cambridge University Press, Cambridge, 1996.

\bibitem{BertoinSSF}
\leavevmode\vrule height 2pt depth -1.6pt width 23pt, {\em Self-similar
  fragmentations}, Ann. Inst. H. Poincar\'e Probab. Statist., 38 (2002),
  pp.~319--340.

\bibitem{BertoinBook}
\leavevmode\vrule height 2pt depth -1.6pt width 23pt, {\em Random fragmentation
  and coagulation processes}, vol.~102 of Cambridge Studies in Advanced
  Mathematics, Cambridge University Press, Cambridge, 2006.

\bibitem{BerFire}
\leavevmode\vrule height 2pt depth -1.6pt width 23pt, {\em Fires on trees},
  Ann. Inst. Henri Poincar\'e Probab. Stat., 48 (2012), pp.~909--921.

\bibitem{BerPercoRTT}
\leavevmode\vrule height 2pt depth -1.6pt width 23pt, {\em Sizes of the largest
  clusters for supercritical percolation on random recursive trees}, Random
  Structures Algorithms, 44 (2014), pp.~29--44.

\bibitem{Ber15}
\leavevmode\vrule height 2pt depth -1.6pt width 23pt, {\em The cut-tree of
  large recursive trees}, Ann. Inst. Henri Poincar\'e Probab. Stat., 51 (2015),
  pp.~478--488.

\bibitem{BCK15}
{\sc J.~Bertoin, N.~Curien, and I.~Kortchemski}, {\em Random planar maps $\&$
  growth-fragmentations},  (2015).
\newblock {P}reprint -- arXiv:1507.02265.

\bibitem{BK15}
{\sc J.~Bertoin and I.~Kortchemski}, {\em Self-similar scaling limits of
  {M}arkov chains on the positive integers}.
\newblock Preprint -- arXiv:1412.1068.

\bibitem{BM13}
{\sc J.~Bertoin and G.~Miermont}, {\em The cut-tree of large {G}alton-{W}atson
  trees and the {B}rownian {CRT}}, Ann. Appl. Probab., 23 (2013),
  pp.~1469--1493.

\bibitem{Berz16+}
{\sc G.~Berzunza}, {\em On scaling limits of multitype {G}alton-{W}atson trees
  with possibly infinite variance},  (2016).
\newblock Preprint -- arXiv:1605.04810.

\bibitem{Bett15}
{\sc J.~Bettinelli}, {\em Scaling limit of random planar quadrangulations with
  a boundary}, Ann. Inst. Henri Poincar\'e Probab. Stat., 51 (2015),
  pp.~432--477.

\bibitem{Bill99}
{\sc P.~Billingsley}, {\em Convergence of Probability Measures}, Wiley, New
  York, 1968.

\bibitem{BGT}
{\sc N.~H. Bingham, C.~M. Goldie, and J.~L. Teugels}, {\em Regular variation},
  vol.~27 of Encyclopedia of Mathematics and its Applications, Cambridge
  University Press, Cambridge, 1989.

\bibitem{BFKV11}
{\sc M.~Bodirsky, {\'E}.~Fusy, M.~Kang, and S.~Vigerske}, {\em Boltzmann
  samplers, {P}\'olya theory, and cycle pointing}, SIAM J. Comput., 40 (2011),
  pp.~721--769.

\bibitem{BDMcLDlS08}
{\sc N.~Broutin, L.~Devroye, E.~McLeish, and M.~de~la Salle}, {\em The height
  of increasing trees}, Random Structures Algorithms, 32 (2008), pp.~494--518.

\bibitem{BrF08}
{\sc N.~Broutin and P.~Flajolet}, {\em The distribution of height and diameter
  in random non-plane binary trees}, Random Structures Algorithms, 41 (2012),
  pp.~215--252.

\bibitem{BM15+}
{\sc N.~Broutin and C.~Mailler}, {\em And/or trees: a local limit point of
  view}, Preprint -- arXiv:1510.06691.

\bibitem{Brw16+}
{\sc N.~Broutin and M.~Wang}, {\em Reversing the cut tree of the brownian
  continuum random tree}, Preprint -- arXiv:1408.2924.

\bibitem{BrW16}
\leavevmode\vrule height 2pt depth -1.6pt width 23pt, {\em Cutting down p-trees
  and inhomogeneous continuum random trees}, Bernoulli, to appear,  (2016).

\bibitem{Cara15+}
{\sc A.~Caraceni}, {\em The {S}caling {L}imit of {O}uterplanar {M}aps}, Annales
  IHP B, to appear.

\bibitem{CFW09}
{\sc B.~Chen, D.~Ford, and M.~Winkel}, {\em A new family of {M}arkov branching
  trees: the alpha-gamma model}, Electron. J. Probab., 14 (2009), pp.~no. 15,
  400--430.

\bibitem{CW13}
{\sc B.~Chen and M.~Winkel}, {\em Restricted exchangeable partitions and
  embedding of associated hierarchies in continuum random trees}, Ann. Inst.
  Henri Poincar\'e Probab. Stat., 49 (2013), pp.~839--872.

\bibitem{CrH10}
{\sc D.~Croydon and B.~Hambly}, {\em Spectral asymptotics for stable trees},
  Electron. J. Probab., 15 (2010), pp.~1772--1801, paper no. 57.

\bibitem{CDKMPref}
{\sc N.~Curien, T.~Duquesne, I.~Kortchemski, and I.~Manolescu}, {\em Scaling
  limits and influence of the seed graph in preferential attachment trees}, J.
  \'Ec. polytech. Math., 2 (2015), pp.~1--34.

\bibitem{CH13}
{\sc N.~Curien and B.~Haas}, {\em The stable trees are nested}, Probab. Theory
  Related Fields, 157(1) (2013), pp.~847--883.

\bibitem{CHK14}
{\sc N.~Curien, B.~Haas, and I.~Kortchemski}, {\em The {CRT} is the scaling
  limit of random dissections}, Random Structures Algorithms, 47 (2015),
  pp.~304--327.

\bibitem{CRLoop13}
{\sc N.~Curien and I.~Kortchemski}, {\em Random stable looptrees}, Electron. J.
  Probab., 19 (2014), pp.~1--35, paper no. 108.

\bibitem{CKPerco}
\leavevmode\vrule height 2pt depth -1.6pt width 23pt, {\em Percolation on
  random triangulations and stable looptrees}, Probab. Theory Related Fields,
  163 (2015), pp.~303--337.

\bibitem{DeRaph15+}
{\sc L.~de~Raph\'elis}, {\em Scaling limit of multitype {G}alton-{W}atson trees
  with infinitely many types}, Annales IHP B, to appear.

\bibitem{Dieul15}
{\sc D.~Dieuleveut}, {\em The vertex-cut-tree of {G}alton-{W}atson trees
  converging to a stable tree}, Ann. Appl. Probab., 25 (2015), pp.~2215--2262.

\bibitem{Dr09}
{\sc M.~Drmota}, {\em Random trees}, SpringerWienNewYork, Vienna, 2009.
\newblock An interplay between combinatorics and probability.

\bibitem{DrG10}
{\sc M.~Drmota and B.~Gittenberger}, {\em The shape of unlabeled rooted random
  trees}, European J. Combin., 31 (2010), pp.~2028--2063.

\bibitem{Duq03}
{\sc T.~Duquesne}, {\em A limit theorem for the contour process of conditioned
  {G}alton-{W}atson trees}, Ann. Probab., 31 (2003), pp.~996--1027.

\bibitem{Duq12}
\leavevmode\vrule height 2pt depth -1.6pt width 23pt, {\em The exact packing
  measure of {L}\'evy trees}, Stochastic Process. Appl., 122 (2012),
  pp.~968--1002.

\bibitem{DLG02}
{\sc T.~Duquesne and J.-F. Le~Gall}, {\em Random trees, {L}\'evy processes and
  spatial branching processes}, Ast\'erisque,  (2002), pp.~vi+147.

\bibitem{DLG05}
\leavevmode\vrule height 2pt depth -1.6pt width 23pt, {\em Probabilistic and
  fractal aspects of {L}\'evy trees}, Probab. Theory Related Fields, 131
  (2005), pp.~553--603.

\bibitem{DLG06}
\leavevmode\vrule height 2pt depth -1.6pt width 23pt, {\em The {H}ausdorff
  measure of stable trees}, ALEA Lat. Am. J. Probab. Math. Stat., 1 (2006),
  pp.~393--415.

\bibitem{DLG09}
\leavevmode\vrule height 2pt depth -1.6pt width 23pt, {\em On the re-rooting
  invariance property of {L}\'evy trees}, Electron. Commun. Probab., 14 (2009),
  pp.~317--326.

\bibitem{DW07}
{\sc T.~Duquesne and M.~Winkel}, {\em Growth of {L}\'evy trees}, Probab. Theory
  Related Fields, 139 (2007), pp.~313--371.

\bibitem{EPW06}
{\sc S.~Evans, J.~Pitman, and A.~Winter}, {\em Rayleigh processes, real trees,
  and root growth with re-grafting.}, Probab. Theory Related Fields, 134(1)
  (2006), pp.~918--961.

\bibitem{Eva08}
{\sc S.~N. Evans}, {\em Probability and real trees}, vol.~1920 of Lecture Notes
  in Mathematics, Springer, Berlin, 2008.
\newblock Lectures from the 35th Summer School on Probability Theory held in
  Saint-Flour, July 6--23, 2005.

\bibitem{Falc14}
{\sc K.~Falconer}, {\em Fractal geometry}, John Wiley \& Sons, Ltd.,
  Chichester, third~ed., 2014.
\newblock Mathematical foundations and applications.

\bibitem{FlS09}
{\sc P.~Flajolet and R.~Sedgewick}, {\em Analytic combinatorics}, Cambridge
  University Press, Cambridge, 2009.

\bibitem{Ford05}
{\sc D.~Ford}, {\em Probabilities on cladograms: introduction to the alpha
  model}.
\newblock Preprint -- arXiv:math/0511246.

\bibitem{GH15}
{\sc C.~Goldschmidt and B.~Haas}, {\em A line-breaking construction of the
  stable trees}, Electron. J. Probab., 20 (2015), pp.~1--24.

\bibitem{HM04}
{\sc B.~Haas and G.~Miermont}, {\em The genealogy of self-similar
  fragmentations with negative index as a continuum random tree}, Electron. J.
  Probab., 9 (2004), pp.~no. 4, 57--97.

\bibitem{HM11}
{\sc B.~Haas and G.~Miermont}, {\em Self-similar scaling limits of
  non-increasing {M}arkov chains}, Bernoulli Journal, 17 (2011),
  pp.~1217--1247.

\bibitem{HM12}
\leavevmode\vrule height 2pt depth -1.6pt width 23pt, {\em Scaling limits of
  {M}arkov branching trees with applications to {G}alton-{W}atson and random
  unordered trees}, Ann. Probab., 40 (2012), pp.~2589--2666.

\bibitem{HMPW08}
{\sc B.~Haas, G.~Miermont, J.~Pitman, and M.~Winkel}, {\em Continuum tree
  asymptotics of discrete fragmentations and applications to phylogenetic
  models}, Ann. Probab., 36 (2008), pp.~1790--1837.

\bibitem{HPW09}
{\sc B.~Haas, J.~Pitman, and M.~Winkel}, {\em Spinal partitions and invariance
  under re-rooting of continuum random trees}, Ann. Probab., 37 (2009),
  pp.~1381--1411.

\bibitem{HSMAP}
{\sc B.~Haas and R.~Stephenson}, {\em Bivariate {M}arkov chains converging to
  {L}amperti transform {M}arkov {A}dditive {P}rocesses}.
\newblock In preparation.

\bibitem{HS15+}
\leavevmode\vrule height 2pt depth -1.6pt width 23pt, {\em Multitype {M}arkov
  {B}ranching trees}.
\newblock In preparation.

\bibitem{HS15}
{\sc B.~Haas and R.~Stephenson}, {\em Scaling limits of {$k$}-ary growing
  trees}, Ann. Inst. Henri Poincar\'e Probab. Stat., 51 (2015), pp.~1314--1341.

\bibitem{Jan06}
{\sc S.~Janson}, {\em Random cutting and records in deterministic and random
  trees}, Random Structures Algorithms, 29 (2006), pp.~139--179.

\bibitem{Janson12}
\leavevmode\vrule height 2pt depth -1.6pt width 23pt, {\em Simply generated
  trees, conditioned {G}alton-{W}atson trees, random allocations and
  condensation: extended abstract}, in 23rd {I}ntern. {M}eeting on
  {P}robabilistic, {C}ombinatorial, and {A}symptotic {M}ethods for the
  {A}nalysis of {A}lgorithms ({A}of{A}'12), Discrete Math. Theor. Comput. Sci.
  Proc., AQ, Assoc. Discrete Math. Theor. Comput. Sci., Nancy, 2012,
  pp.~479--490.

\bibitem{JanStef15}
{\sc S.~Janson and S.~{\"O}. Stef{\'a}nsson}, {\em Scaling limits of random
  planar maps with a unique large face}, Ann. Probab., 43 (2015),
  pp.~1045--1081.

\bibitem{Kolmo41}
{\sc A.~Kolmogorov}, {\em Uber das logarithmisch normale {V}erteilungsgesetz
  der {D}imensionen der {T}eilchen bei {Z}erstuckelung}, C.R. Acad. Sci.
  U.R.S.S., 31 (1941), pp.~99--101.

\bibitem{Kort12}
{\sc I.~Kortchemski}, {\em Invariance principles for {G}alton-{W}atson trees
  conditioned on the number of leaves}, Stochastic Process. Appl., 122 (2012),
  pp.~3126--3172.

\bibitem{KortSimple}
\leavevmode\vrule height 2pt depth -1.6pt width 23pt, {\em A simple proof of
  {D}uquesne's theorem on contour processes of conditioned {G}alton-{W}atson
  trees}, in S\'eminaire de {P}robabilit\'es {XLV}, vol.~2078 of Lecture Notes
  in Math., Springer, Cham, 2013, pp.~537--558.

\bibitem{LG06}
{\sc J.-F. Le~Gall}, {\em Random real trees}, Ann. Fac. Sci. Toulouse Math.
  (6), 15 (2006), pp.~35--62.

\bibitem{LGLJ98}
{\sc J.-F. Le~Gall and Y.~Le~Jan}, {\em Branching processes in {L}\'evy
  processes: the exploration process}, Ann. Probab., 26 (1998), pp.~213--252.

\bibitem{Mar03}
{\sc P.~Marchal}, {\em Constructing a sequence of random walks strongly
  converging to {B}rownian motion}, in Discrete random walks ({P}aris, 2003),
  Discrete Math. Theor. Comput. Sci. Proc., AC, Assoc. Discrete Math. Theor.
  Comput. Sci., Nancy, 2003, pp.~181--190).

\bibitem{Mar06}
\leavevmode\vrule height 2pt depth -1.6pt width 23pt, {\em A note on the
  fragmentation of a stable tree}, in Fifth {C}olloquium on {M}athematics and
  {C}omputer {S}cience, Discrete Math. Theor. Comput. Sci. Proc., AI, Assoc.
  Discrete Math. Theor. Comput. Sci., Nancy, 2008, pp.~489--499.

\bibitem{MM11}
{\sc J.-F. Marckert and G.~Miermont}, {\em The {CRT} is the scaling limit of
  unordered binary trees}, Random Structures Algorithms, 38 (2011),
  pp.~467--501.

\bibitem{MiermontStFlour}
{\sc G.~Miermont}, {\em Aspects of Random Maps, Lecture notes of the 2014
  Saint--Flour Probability Summer School}.
\newblock Preliminary draft:
  http://perso.ens-lyon.fr/gregory.miermont/coursSaint-Flour.pdf.

\bibitem{Mier03}
{\sc G.~Miermont}, {\em Self-similar fragmentations derived from the stable
  tree. {I}. {S}plitting at heights}, Probab. Theory Related Fields, 127
  (2003), pp.~423--454.

\bibitem{Mier05}
\leavevmode\vrule height 2pt depth -1.6pt width 23pt, {\em Self-similar
  fragmentations derived from the stable tree. {II}. {S}plitting at nodes},
  Probab. Theory Related Fields, 131 (2005), pp.~341--375.

\bibitem{Mier08}
\leavevmode\vrule height 2pt depth -1.6pt width 23pt, {\em Invariance
  principles for spatial multitype {G}alton-{W}atson trees}, Ann. Inst. Henri
  Poincar\'e Probab. Stat., 44 (2008), pp.~1128--1161.

\bibitem{Otter48}
{\sc R.~Otter}, {\em The number of trees}, Ann. of Math. (2), 49 (1948),
  pp.~583--599.

\bibitem{Pagnard16}
{\sc C.~Pagnard}, {\em Local limit and volume growth of {M}arkov-{B}ranching
  trees}.
\newblock In preparation.

\bibitem{PS16+}
{\sc K.~Panagiotou and B.~Stufler}, {\em Scaling limits of random {P}ólya
  trees}, Preprint -- arXiv:1502.07180.

\bibitem{PSW15+}
{\sc K.~Panagiotou, B.~Stufler, and K.~Weller}, {\em Scaling limits of graphs
  from subcritical classes}, Preprint, to appear in Ann. of Probab.

\bibitem{Pan06}
{\sc A.~Panholzer}, {\em Cutting down very simple trees}, Quaest. Math., 29
  (2006), pp.~211--227.

\bibitem{Pavlov77}
{\sc J.~L. Pavlov}, {\em The asymptotic distribution of the maximum size of
  trees in a random forest}, Teor. Verojatnost. i Primenen., 22 (1977),
  pp.~523--533.

\bibitem{PitCoal99}
{\sc J.~Pitman}, {\em Coalescent random forests}, J. Combin. Theory Ser. A, 85
  (1999), pp.~165--193.

\bibitem{PitmanStFl}
{\sc J.~Pitman}, {\em Combinatorial stochastic processes}, vol.~1875 of Lecture
  Notes in Mathematics, Springer-Verlag, Berlin, 2006.
\newblock Lectures from the 32nd Summer School on Probability Theory held in
  Saint-Flour, July 7--24, 2002.

\bibitem{Polya37}
{\sc G.~P{\'o}lya}, {\em Kombinatorische {A}nzahlbestimmungen f\"ur {G}ruppen,
  {G}raphen und chemische {V}erbindungen}, Acta Math., 68 (1937), pp.~145--254.

\bibitem{Rem85}
{\sc J.-L. R{\'e}my}, {\em Un proc\'ed\'e it\'eratif de d\'enombrement d'arbres
  binaires et son application \`a leur g\'en\'eration al\'eatoire}, RAIRO
  Inform. Th\'eor., 19 (1985), pp.~179--195.

\bibitem{Riz15}
{\sc D.~Rizzolo}, {\em Scaling limits of {M}arkov branching trees and
  {G}alton-{W}atson trees conditioned on the number of vertices with out-degree
  in a given set}, Ann. Inst. H. Poincar\'e Probab. Statist., 51 (2015),
  pp.~512--532.

\bibitem{Ste09}
{\sc S.~Stef\'ansson}, {\em The infinite volume limit of {F}ord’s alpha model},
  Acta Physica Polonica B Proc. Suppl. 2, 3 (2009), p.~555–560.

\bibitem{Steph13}
{\sc R.~Stephenson}, {\em General fragmentation trees}, Electron. J. Probab.,
  18(101) (2013), pp.~1--45.

\bibitem{Stu16+}
{\sc B.~Stufler}, {\em The continuum random tree is the scaling limit of
  unlabelled unrooted trees}, Preprint -- arXiv:1412.6333.

\bibitem{Stu16+2}
\leavevmode\vrule height 2pt depth -1.6pt width 23pt, {\em Random enriched
  trees with applications to random graphs}, Preprint -- arXiv:1504.02006.

\end{thebibliography}

\end{document}